\documentclass[reqno,11pt]{amsart}
\usepackage{amsmath,amssymb,amsthm,bm,amsfonts,mathrsfs,graphicx}
\usepackage{breqn}
\usepackage[mathcal]{euscript}
\usepackage[breaklinks=true]{hyperref}
\usepackage{cleveref,enumitem}
\usepackage{macrostxd}
\usepackage[english]{babel}
\usepackage{setspace}

\usepackage{latexsym,bm}
\usepackage{amscd}
\usepackage{epstopdf}
\usepackage[all]{xypic}
\usepackage[titletoc]{appendix}
\usepackage{xfrac,xcolor}
\usepackage[normalem]{ulem}

\newcommand{\set}[2]{\left\{#1:~#2\right\}}

\newcommand\numberthis{\addtocounter{equation}{1}\tag{\theequation}}

    \makeatletter
    \def\step{
    \@ifnextchar[ \@step{\@noitemargtrue\@step[\@itemlabel]}}
    \def\@step[#1]{
        \item[#1]\textit{}\hspace*{\dimexpr-\labelwidth-\labelsep}
        }
    \makeatother
\setenumerate{fullwidth,itemindent=\parindent,
    listparindent=\parindent,itemsep=0ex,
    partopsep=0pt,parsep=0ex}

\begin{document}

\onehalfspacing

\title[Instability of the solitary waves of gDNLS in the degenerate case]
{Instability of the solitary waves for the generalized derivative nonlinear Schr\"{o}dinger
equation in the degenerate case}

\author[]{Changxing Miao}
\address{\hskip-1.15em Changxing Miao:
\hfill\newline Institute of Applied Physics and Computational
Mathematics, \hfill\newline P. O. Box 8009,\ Beijing,\ China,\
100088.}  \email{miao\_changxing@iapcm.ac.cn}

\author[]{Xingdong Tang}
\address{\hskip-1.15em Xingdong Tang \hfill\newline Beijing Computational Science Research Center, \hfill\newline No. 10 West Dongbeiwang Road, Haidian
District, Beijing, China, 100193,}
\email{xdtang202@163.com}

\author[]{Guixiang Xu}
\address{\hskip-1.15em Guixiang Xu \hfill\newline Institute of
Applied Physics and Computational Mathematics, \hfill\newline P. O.
Box 8009,\ Beijing,\ China,\ 100088.}
\email{xu\_guixiang@iapcm.ac.cn}



\begin{abstract}
In this paper, we develop the modulation analysis, the perturbation argument and
the Virial identity similar as those in \cite{MartelM:Instab:gKdV}
to  show the orbital instability of the solitary waves $\Q\sts{x-ct}\e^{\i\omega t}$
of the generalized derivative nonlinear Schr\"{o}dinger equation (gDNLS)
in the degenerate case $c=2z_0\sqrt{\omega}$, where $z_0=z_0\sts{\sigma} $
is the unique zero point of $F\sts{z;~\sigma}$ in $\sts{-1, ~ 1}$.
The new ingredients in the proof are the refined modulation decomposition of the solution
near $\Q$ according to the spectrum property of the linearized operator
$\Scal_{\omega, c}''\sts{\Q}$ and the refined construction of
the Virial identity in the degenerate case. Our argument is qualitative,
and we improve the result in \cite{Fukaya2017}.
\end{abstract}

\maketitle


\section{Introduction}
In this paper, we consider the generalized
derivative nonlinear Schr{\"o}dinger (gDNLS) equation
\begin{equation}
\label{gdnls}
\left\lbrace
\aligned
 &\i {u}_{t} +{u}_{xx}+\i \abs{u}^{2\sigma}{u}_{x}= \; 0,
 \quad\sts{t,x}\in\R \times\R,
 \\
 &u\sts{0,x}= u_0\sts{x} \in  H^1(\R),
\endaligned
\right.
\end{equation}
where $1<\sigma<2$. The equation \eqref{gdnls} is $\dot H^{\frac{\sigma-1}{2\sigma}}$-critical,  and therefore $L^2$-supercritical for $\sigma\in \sts{1,~2}$, since the scaling
transformation
\begin{align*}
  u(t,x)\mapsto
  u_{\lambda}(t,x)
  \triangleq
  \lambda^{\frac{1}{2\sigma}}u(\lambda^2t, \lambda x)
\end{align*}
leaves both \eqref{gdnls} and $\dot H^{\frac{\sigma-1}{2\sigma}}$-norm invariant. \eqref{gdnls} appears in plasma physics (see \cite{MOMT-PHY, M-PHY,
SuSu-book}) and as a model for ultrashort optical pulses for $\sigma=1$ (see \cite{MMW-PHY}).

By the Picard iteration argument with the Strichartz estimates in \cite{Caz:NLS:book}, local well-posedness for \eqref{gdnls} with $\sigma \in \sts{1, 2}$ in the energy space $H^1\sts{\R}$ is
now well understood by Hayashi and Ozawa in \cite{HaOz2016}. More precisely, for any $u_0\in H^1\sts{\R}$, there exists $0<T_{\max}\leq+\infty$ and a unique solution $u\in \Ccal\sts{[0,T_{\max}),~H^1\sts{\R}}$ of \eqref{gdnls}. Moreover, the mass, the momentum and the energy are conserved under the flow \eqref{gdnls}, i.e. for any $t\in [0,T_{\max})$,
\begin{equation}\label{M}	\Mcal\sts{u\sts{t}}
=
\frac{1}{2}\int_{\R}\abs{u\sts{t,x}}^2\d x
=
\Mcal\sts{u_0},
\end{equation}
\begin{equation}\label{P}
\Pcal\sts{u\sts{t}}
=
\frac{1}{2}\Re\int_{\R} \i\sts{\bar{u}{u}_{x}}\sts{t,x}\d x
=
\Pcal\sts{u_0},
\end{equation}
\begin{equation}\label{E}
\Ecal\sts{u\sts{t}}
=
\frac{1}{2}\int_{\R}\abs{{u}_{x}\sts{t,x}}^2\d x
-
\frac{1}{2\sigma + 2}
\Re\int_{\R}\i\sts{\abs{u}^{2\sigma}\bar{u}{u}_{x}}\sts{t,x}\d x
=
\Ecal\sts{u_0}.
\end{equation}

In addition, it was observed in \cite{KaupN:DNLS:soliton, LSS2013} that
the equation \eqref{gdnls} has
a two-parameter family of solitary wave solutions
with the following form
\begin{equation}
\label{uQ}
u\sts{t,x}=\Q\sts{x-ct}\e^{\i\omega t},
\end{equation}
where $4\omega > c^2$ and
\begin{align}
\label{Q}
  Q_{\omega,c}\sts{x}=\Psi_{\omega,c}(x)\exp \left\{
    \i \frac{c}{2}x-\frac{\i}{2\sigma +
      2}\int^{x}_{-\infty}\Psi_{\omega,c}^{2\sigma}(y)\d y\right\}
\end{align}
with
\begin{equation}
  \label{Phi}
  \Psi_{\omega,c}(x)
  =
  \sts{
    \frac{(\sigma+1)(4\omega-c^2)}{
        2\sqrt{\omega}(\cosh(\sigma\sqrt{4\omega-c^2}x)-\frac{c}{2\sqrt{\omega}})
        }
    }^{\frac{1}{2\sigma}}.
\end{equation}
In fact, for the case $\sigma\in \sts{1,~2}$ and $4\omega> c^2$, $ \Psi_{\omega,c}(x)$ is the unique positive solution of
\begin{equation*}
  - \partial_{x}^2\Psi_{\omega,c}
  +
  (\omega -
  \frac{c^2}{4})
  \Psi_{\omega,c} +
  \frac{c}{2}|\Psi_{\omega,c}|^{2\sigma}\Psi_{\omega,c}
  - \frac{2\sigma + 1}{(2\sigma +
    2)^2}
    |\Psi_{\omega,c}|^{4\sigma}\Psi_{\omega,c}  = 0,
\end{equation*}
up to phase rotation and spatial translation symmetries. Moreover, by the expression of $\Q$, we have
\begin{equation}
\label{QDQ}
  c_0
  \leq
  \left|\frac{\abs{ \partial_{x}{\Q}\sts{x}} }{\Q\sts{x}}\right|
  \leq
  \frac{1}{c_0},
\end{equation}
where $0<c_0<1$ is a constant independent of $x$, and by inserting \eqref{uQ} into \eqref{gdnls}, we have
\begin{equation}
\label{Qeq}
-\partial_{x}^{2}{\Q}+\omega~\Q+c\i\partial_{x}{\Q}-\i\abs{\Q}^{2\sigma}\partial_{x}{\Q}=0.
\end{equation}


For the case $\sigma\in(1,2)$ and $4\omega > c^2$,
Fukaya, Hayashi and Inui recently made use of the structure analysis to show the variational characterization of $\Q$ in \cite{FukHI:gDNLS}, i.e.
$Q_{\omega,c}$ is a minimizer of the following variational problem:
\begin{align}
\label{d}
    d\sts{\omega,c}
    =
    \inf\set{\Scal_{\omega,c}\sts{\varphi} }{ \varphi\in H^1\setminus\ltl{0},
    ~~ \mathcal{K}_{\omega,c}\sts{\varphi}=0  }
\end{align}
where the action functional $\Scal_{\omega,c}\sts{\varphi}$ is defined by
\begin{equation}
  \label{S}
  \Scal_{\omega,c}(\varphi)=\Ecal(\varphi)+ \omega ~\Mcal(\varphi) + c ~\Pcal(\varphi),
\end{equation}
and the scaling derivative functional $  \mathcal{K}_{\omega,c}(\varphi)$ is
defined by
\begin{equation}
  \label{K}
  \mathcal{K}_{\omega,c}(\varphi)
  =
  \left.\frac{\d}{\d \lambda}\Scal_{\omega,c}(\lambda\varphi)\right|_{\lambda=1}.
\end{equation}
In addition, a sufficient condition of the global wellposedness for \eqref{gdnls}
in the energy space was also induced by the variational characterization of
$\Q$ in \cite{FukHI:gDNLS}.
It is worthy that the authors firstly made use of the structure analysis
to show the variational characterization of $\Q$ with $\sigma=1$ and $4\omega > c^2$
in \cite{MiaoTX:Exist}. We can also refer to
\cite{AmbMal:book, IbMasN:NLKG:scat, Willem:book}
and references therein for the variational characterization of the solitary waves.

We now recall the definition of the orbital stability
in order to show the orbital instability analysis of the solitary waves.

\begin{defi}
Let $U(t,x)$ be a solitary wave solution of \eqref{gdnls}.
We say that $U(t,\cdot)$ is orbitally stable (up to phase rotation and spatial translation symmetries)
if for any $\epsilon>0$, there exists $\delta>0$
such that if $u_0\in H^1(\R)$ with $\big\|u_0\sts{\cdot}-U(0,\cdot)\big\|_{H^1}<\delta$,
then the solution $u(t)$ of \eqref{gdnls} with
initial data $u_0$ exist globally in time and satisfies
\begin{equation*}
\sup_{t\geq 0} \inf_{(y, \gamma)\in \R^2}
\big\|u(t,\cdot)-U(t, \cdot-y)
e^{i\gamma}
\big\|_{H^1(\R)}<\epsilon.
\end{equation*}
Otherwise, $U(t,\cdot)$ is said to be orbitally unstable. 
\end{defi}
By means of the classical stability theory of the solitary waves of the Hamiltonian PDEs in \cite{GSS1987, GSS1990, MartelMT:Stab:gKdV, MartelMT:Stab:NLS, Pave:book} and references therein,  the crucial idea in the proof of the orbital stability of the solitary waves $\Q\sts{x-ct}\e^{\i\omega t}$ of \eqref{gdnls} with $\sigma=1,~4\omega > c^2$ in \cite{MiaoTX:Stab} is essentially to show that $\Q$ is a local minimizer of the action functional $\Scal_{\omega, c}\sts{u}$ (or the energy functional $\Ecal\sts{u}$)
over the set of all admissible functions $u$ satisfying
\begin{align*}
\Mcal\sts{u} = \Mcal\sts{\Q}, \;\; \Pcal\sts{u} = \Pcal\sts{\Q}
\end{align*}
by the non-degenerate condition of the Hessian matrix
\begin{equation}
\label{ddp}
d''\sts{\omega,c}
 \triangleq
\begin{bmatrix}
	\partial_{\omega}^{2}{d\sts{\omega,c}} & \partial_{\omega,c}{d\sts{\omega,c}}
	\\[6pt]
	\partial_{c,\omega}{d\sts{\omega,c}} & \partial_{c}^{2}{d\sts{\omega,c}}
\end{bmatrix}
\end{equation}
of the function
$d\sts{\omega, c} =   S_{\omega,c}\sts{\Q}$
for the case $\sigma=1,~4\omega > c^2$.
In fact, there is another nonlinear argument based on
the concentration compactness principle to show the orbital stability
of the solitary waves of Hamiltonian system, we can refer to
\cite{CaL:NLS:stable, ColinOhta-DNLS} and references therein.

While for the equation \eqref{gdnls} when $\sigma \in \sts{1,~2},~4\omega > c^2$, it becomes complicate because the degeneracy of the Hessian matrix $d''\sts{\omega,c}$ may occur as shown in \cite{LSS2013}. More precisely, the degeneracy of the Hessian matrix $d''\sts{\omega,~c}$ occurs for the case $\sigma \in \sts{1,~2},~4\omega > c^2$ if and only if $c=2z_0 \sqrt{\omega}$, where $z_0=z_0(\sigma)$ is the unique zero point in \(\sts{-1,1}\) of the following function
\begin{multline}\tiny
         F(z; \sigma) \\
        = (\sigma-1)^2 \ltl{\int_0^\infty (\cosh y -z)^{-\frac{1}{\sigma}} \d y}^2
        - \ltl{\int_0^\infty (\cosh y -z)^{-\frac{1}{\sigma}-1}(z
          \cosh y -1)\d y }^2.\label{eq:sig-z}
\end{multline}
For the case $\sigma \in \sts{1,~2}$, please refer to Figure \ref{fig:Fcurves} for the distribution of the zero point $z_0\sts{\sigma}$ of $F(z; \sigma)$ in $\sts{-1,~1}$, this figure comes from \cite{LSS2013}.
%
%
\begin{figure}[ht]
\includegraphics[width=8cm]{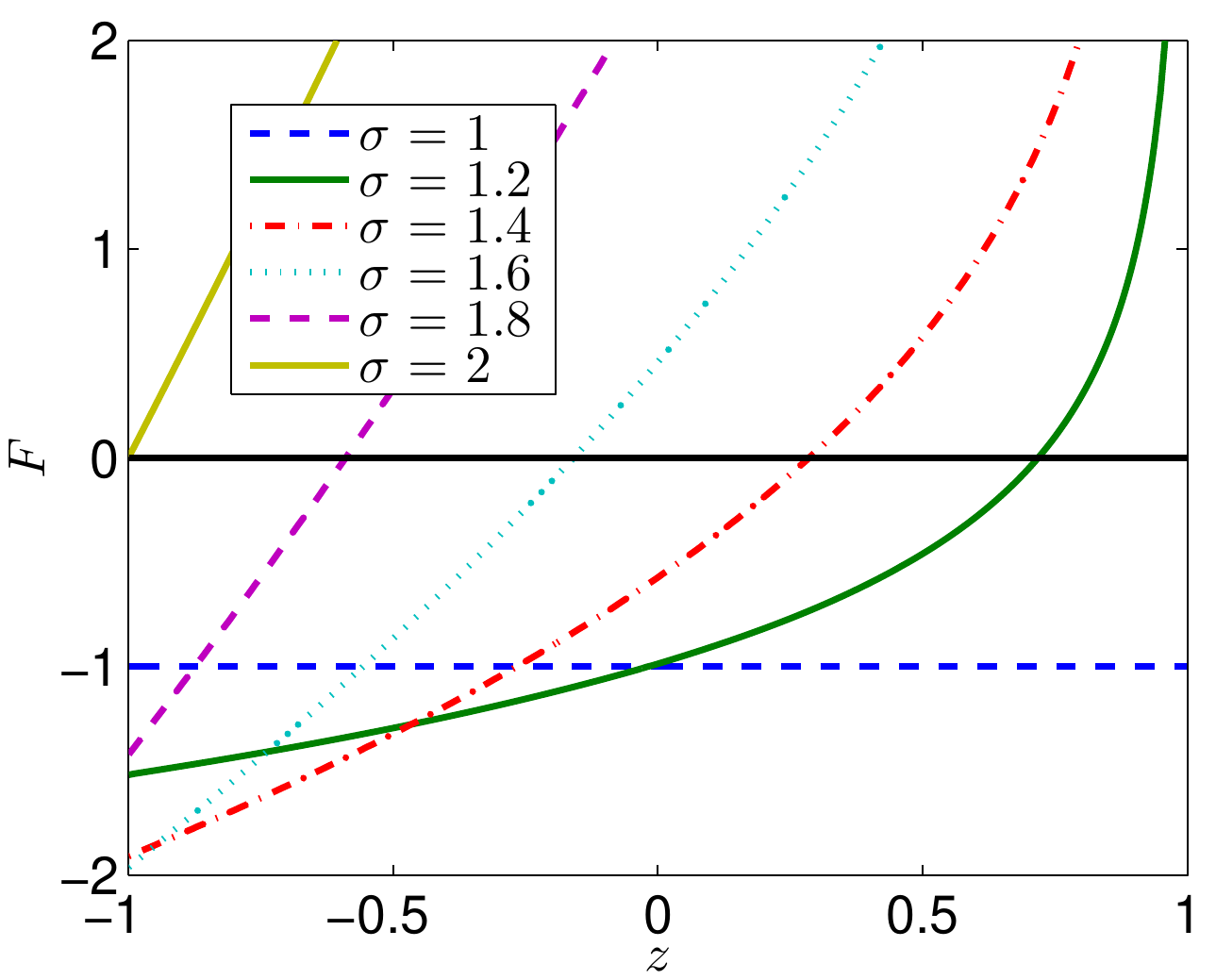}
\caption{The zero point $z_0\sts{\sigma}$ of $F(z; \sigma)$ in $\sts{-1,~1}$ for $\sigma \in \sts{1,~2}$, and $z_0\sts{\sigma}$ is a decreasing function as $\sigma $ increases in $(1, 2)$. }\label{fig:Fcurves}
\end{figure}

At the same time, by means of the stability criteria in \cite{GSS1990}, Liu, Simpson and Sulem numerically showed that the solitary waves $\Q\sts{x-ct}\e^{\i\omega t}$  of \eqref{gdnls} with $\sigma\in \sts{1,~2}$ are orbitally stable for the case $c\in (-2\sqrt{\omega}, 2z_0\sqrt{\omega})$  and orbitally unstable for the case $c\in (2z_0\sqrt{\omega}, 2\sqrt{\omega})$ in \cite{LSS2013}.  Recently, in \cite{TX:Stab}, the last two authors  made use of the modulation analysis, the perturbation argument and the energy argument as those in \cite{MartelMT:Stab:gKdV, MartelMT:Stab:NLS} (also see \cite{LeWu:DNLS, MiaoTX:Stab}), to show the orbital stability of the sum of the two-soliton waves of \eqref{gdnls} with $c_k\in (-2\sqrt{\omega_k}, ~2z_0\sqrt{\omega_k}),~k=1,~2$ under some technical conditions.

As for the degenerate case $\sigma\in (1, 2)$ and $c=2z_0 \sqrt{\omega}$,
both stability analysis and instability analysis in \cite{GSS1987, GSS1990} fail because of the degeneracy of the Hessian matrix $d''\sts{\omega, ~c}$. In fact, it was shown that the degeneracy is of finite order in \cite{LSS2013}, and there exists a vector ${\bm\xi}=\sts{\xi_1, \xi_2}^{\text{T}}\in \R^2\backslash\{\bm 0\}$ such that
    \begin{align*}
    d''\sts{\omega,c} \cdot {\bm\xi} = {\bm 0}, \;\;\text{and}\; \; \mathbf{d}_{\bm\xi}^{\tprime}
=
  \left.
  \frac{\d^3}{\d \lambda^3}
  d\sts{\omega+\lambda\xi_1,c+\lambda\xi_2}
  \right|_{\lambda=0}
  \neq
  0.
    \end{align*}
Now let $\widetilde{\varphi}_{\omega, c}$ be defined by
$
\widetilde{\varphi}_{\omega, c}\sts{x}= \xi_1 ~\partial_{\omega}\Q\sts{x} + \xi_2~ \partial_{c}\Q\sts{x},
$
and
$\varphi_{\omega, c}\sts{x}$ be defined by
\begin{align*}
  \varphi_{\omega, c}\sts{x}
=
  {\tphi}_{\omega, c}
  -
  a~
  {\partial_x \Q}\sts{x}
  -
  b~
  {\i{\Q}}\sts{x}
\end{align*}
where $a$ and $b$ are
\begin{equation}\label{quan:ab}
a
  =
  \frac{
    \det\begin{bmatrix}
           \action{\i {\Q}}{\tphi_{\omega, c}} & \action{\i{\Q}}{\i{\Q}} \\
            \action{\partial_x{\Q}}{\tphi_{\omega,c}} & \action{\partial_x{\Q}}{\i{\ Q}}\\
         \end{bmatrix}
  }{
    \det\begin{bmatrix}
           \action{\i {\Q}}  {\partial_x \Q}  & \action{\i{\Q}}{\i{\Q}} \\
          \action{\partial_x{\Q}}  {\partial_x \Q}  & \action{\partial_x{\Q}}{\i{\Q}}\\
         \end{bmatrix}
  },
\quad
b
  =
  \frac{
    \det\begin{bmatrix}
         \action{\i {\Q}}  {\partial_x \Q}  & \action{\i {\Q}}{\tphi_{\omega, c}} \\
       \action{\partial_x{\Q} }  {\partial_x \Q} & \action{\partial_x{\Q} }{\tphi_{\omega, c}}\\
         \end{bmatrix}
  }{
    \det\begin{bmatrix}
           \action{\i {\Q}}  {\partial_x \Q}  & \action{\i{\Q}}{\i{\Q}} \\
          \action{{\Q}_{x}}  {\partial_x \Q}  & \action{\partial_x{\Q} }{\i{\Q}}\\
         \end{bmatrix}
  }.
\end{equation}
For $u\in H^1$, let $\Jcal\sts{u}$ be defined by $
\Jcal\sts{u}= \omega~\Mcal\sts{u} + c~ \Pcal\sts{u},
$ and $\Bcal\Q\sts{x}$ be defined by
\begin{align*}\Bcal\Q\sts{x}=\xi_1~\Q\sts{x} + \xi_2~\i\partial_x\Q\sts{x}.
\end{align*}
Based on the above preparations,
we can state the following orbital instability result
of the solitary waves of \eqref{gdnls} in the degenerate case.

\begin{theo}\label{thm}
  Let $\sigma\in \sts{1, 2}$, and $c=2 z_0 \sqrt{\omega}$
  where $z_0=z_0(\sigma)$ be the unique zero point in \(\sts{-1,1}\) of $F(z; \sigma)$
  in \eqref{eq:sig-z}, then the solitary wave
$\Q\sts{x-ct}\e^{\i\omega t}$ of \eqref{gdnls} is orbitally unstable.
More precisely, there exist $\alpha_0>0$ and $\lambda_0>0$ such that if
 $$u_{0}\sts{x}= \Q\sts{x}+ \lambda~\phi_{\omega, c} \sts{x} +\widetilde{\rho}\sts{\lambda}~\B{\Q}\sts{x}$$
where $0<\lambda<\lambda_0$ and $\widetilde{\rho}\sts{\lambda}$ is chosen such that
$\Jcal\sts{u_{0}}=\Jcal\sts{\Q}$,
then there exists $t_0=t_0(u_{0})$ such that the solution
$u\sts{t}$ of \eqref{gdnls} with initial data $u_{0}$ satisfies
\begin{equation*}
 \inf_{(y, \gamma)\in \R^2}\big\|u(t_0,\cdot)
 -
 \Q\sts{\cdot-ct_0-y}\e^{\i\omega t_0+\i\gamma}\big\|_{H^1(\R)} \geq \alpha_0.
\end{equation*}
\end{theo}

\begin{rema} Here we give some remarks related to the above result.
\begin{enumerate}
\item There is  nonlinear restriction on $u_0$ (or $\widetilde{\rho}(\lambda)$)
through the functional $\Jcal\sts{u}$ in the assumption,
it is reasonable by the Implicit Function Theorem (see  Lemma \ref{lem:d1}). In fact, $ \widetilde{\rho}$ is a $\mathcal{C}^1$ function and can be taken by
\begin{equation*}
  \widetilde{\rho}\sts{\lambda} = -\frac{\action{\B{\phi_{\omega, c}}}{\phi_{\omega, c}} }{2\action{\B{\Q}}{\B{\Q}}}\lambda^2 +\so{\lambda^2}.
  \end{equation*}

\item For the mass critical gKdV equation. In \cite{MartelM:Instab:gKdV},
Martel and Merle combined the modulation analysis, the perturbation argument
and the Kato-Virial identity with the pointwise decay estimate of
the linear KdV flow to show the orbital instability of the traveling waves.
We give here the refined decomposition for the function near
$\Q$ (see Lemma \ref{lem:d2}), which helps us to understand
the refined landscape of the action functional $\Scal_{\omega,c}$ around $\Q$ .
It turns out that even though up to the phase rotation and spatial translation
symmetries, $\Q$ is not the local minimizer of the action functional
$\Scal_{\omega, c}$ any more, and $\Scal_{\omega, c}$ is a locally monotone functional
along the direction $\varphi_{\omega,c}$ at $\Q$
(see Lemma \ref{lem:tyl1} and Lemma \ref{lem:tyl2}).

\item Comech and Pelinovsky proved nonlinear instability of the standing waves
with minimal energy of Hamiltonian system with $U(1)$ symmetry in \cite{CoPe2003},
which was caused by higher order algebraic degeneracy of
the zero eigenvalue in the spectrum of the linearized system.
Later, Ohta \cite{Ohta:Ins:JFA} and Maeda  \cite{Maeda2012}
shown the criterion of the orbital instability and
stability of bound states in the finite degenerate case under the framework
of Grillakis, Shatah and Strauss's argument in \cite{GSS1987, GSS1990} successively.
Compared with these arguments, our modulation decomposition is related to
the Hamitonian struture and the monotonicity formula comes from the
dynamical behavior of the radiation term $\eps\sts{t}$
(i.e., the Virial identity, see \eqref{I} and \eqref{eIt}) in this paper.

\item For the case $\sigma\in [3/2,~2)$ and $c=2z_0 \sqrt{\omega}$, Fukaya made use of the argument in \cite{Ohta:Ins:JFA} with the fact that $E\in \Ccal^3 \sts{H^1(\R),~\R}$ to show that the solitary wave $\Q\sts{x-ct}\e^{\i\omega t}$ of \eqref{gdnls} is orbitally unstable in \cite{Fukaya2017}. It turns out that the condition $E\in \Ccal^3 \sts{H^1(\R),~\R}$ is not a necessary condition, the local condition that $E\in \Ccal^3 $ at $\Q$, which can be ensured by the positivity property of $\abs{\Q}$,  is enough for us to show the orbital instability of the solitary wave of \eqref{gdnls} in the degenerate case $\sigma\in \sts{1,~2}$ and $c=2z_0 \sqrt{\omega}$.

\end{enumerate}

\end{rema}

As stated above, the classical modulation analysis and the Virial identity in \cite{GSS1987, GSS1990, Shatah:NLKG:unstab, Shatah:NLS:unstab} doesn't work once again because of the degenerate property of the Hessian matrix $d''\sts{\omega, c}$ for the case $\sigma\in \sts{1,~2}$ and $c=2z_0 \sqrt{\omega}$, we now give more explanations about the refined modulation analysis and the refined Virial identity.


Firstly, we use the following decomposition
\begin{equation}\label{decomp_v2}
 u\sts{x} = \e^{-\i \gamma }
     \bsts{
    \Q
    +
    {\lambda}\phi_{\omega, c}
    +
    \rho({ \lambda})\B{\Q}  + \eps
    }\sts{ x-y },  \;\; \rho\sts{\lambda} = -\frac{\action{\B{\phi_{\omega, c}}}{\phi_{\omega, c}} }{2\action{\B{\Q}}{\B{\Q}}}\cdot \lambda^2
\end{equation}
for the function $u$ in the $\delta$-tube $ \Ucal\sts{\Q~,~\delta} $ of $\Q$
(see \eqref{tube} for the definition of the $\delta$-tube of $\Q$ and
the directions $\widetilde{\varphi}_{\omega, c}$, $\varphi_{\omega, c}$ and
$\Bcal \Q$ in Figure \ref{fig:decomp}), the above refined decomposition is
related with the landscape of the action functional $\Scal_{\omega, c}$ near $\Q$.
%
%
\vskip-0.1in
\begin{figure}[ht]
\includegraphics[width=8.5cm]{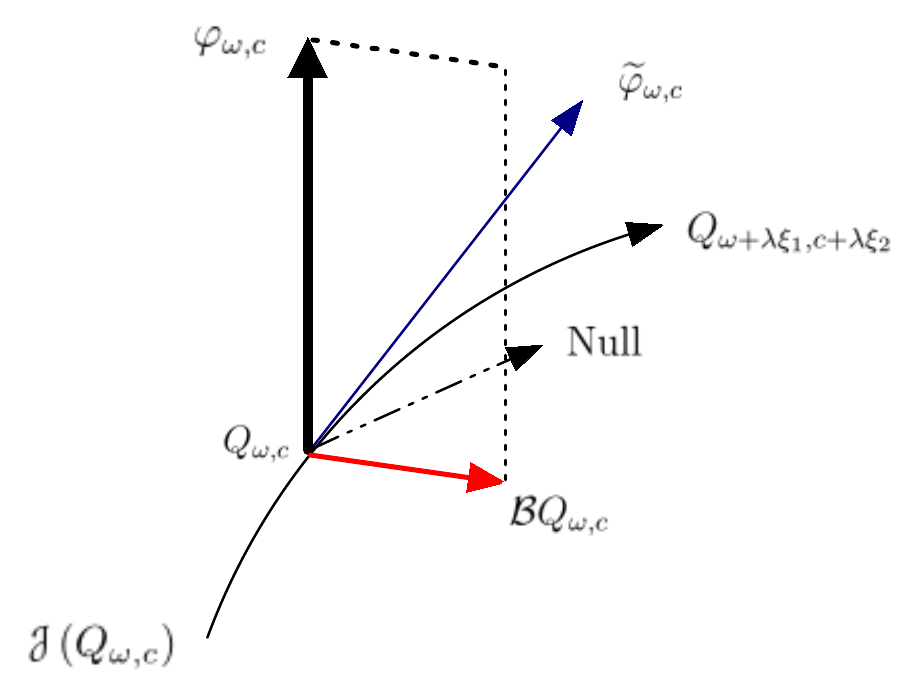}
\caption{ Decompostion of the $\delta$-tube $\Ucal\sts{\Q, ~\delta}$ up to small radiation $\eps$.}\label{fig:decomp}
\end{figure}
\vskip-0.1in

\begin{enumerate}
\item By the variational characterization of $\Q$, the action functional
$\Scal_{\omega,c} $
has the following properties
\begin{align*}
\Scal'_{\omega,c}\sts{\Q} =0,{\;\; \Scal''_{\omega,c}\sts{\Q} = \mathcal{L},}
\end{align*}
where the null space of the linearized operator
${\mathcal L}$
is characterized by
$\text{Null}\sts{\mathcal L}=\text{span}\{\i\Q, \partial_x \Q\}$.

By the finite degenerate property of the Hessian matrix of the function $d\sts{\omega, c}=\Scal_{\omega, c}\sts{\Q}$ for the case $\sigma\in\sts{1,~2}$ and  $c=2z_0\sqrt{\omega}$ in \cite{LSS2013},  there exists a direction ${\bm\xi}=\sts{\xi_1, \xi_2}^{\text{T}}\in \R^2\backslash \{0\}$ such that
    \begin{align*}
    d''\sts{\omega,c} \cdot {\bm\xi} =0, \;\;\text{and}\; \; \mathbf{d}_{\bm\xi}^{\tprime}
  \triangleq
  \left.
  \frac{\d^3}{\d \lambda^3}
  d\sts{\omega+\lambda\xi_1,c+\lambda\xi_2}
  \right|_{\lambda=0}
  \neq
  0,
    \end{align*}
where the first equality means that the quantity
 $$\Jcal_{\bm\xi}\sts{u}
=\xi_1\Mcal\sts{u}+\xi_2\Pcal\sts{u}$$
has the local equilibrium point $\Q$  along the curve $\{Q_{\omega + \lambda \xi_1, c+\lambda \xi_2}\}_{\lambda\in \R}$.
\item Up to the symmetries (spatial translation and phase rotation invariances), the first order approximation of $u$ to $\Q$ comes from the tangent vector $\widetilde{\varphi}_{\omega, c}$ of the curve $\{Q_{\omega+\lambda \xi_1, c+\lambda \xi_2}\}_{\lambda\in \R}$  at $\Q$,  and we have following degenerate result
    \begin{align}\label{dir_xi}
     \Scal^{\dprime}_{\omega, c}\sts{\Q}
    \sts{\widetilde{\phi}_{\omega, c}, ~\widetilde{\phi}_{\omega, c}} = 0.
    \end{align}

\item Up to the symmetries (spatial translation and phase rotation invariances), the second order approximation of $u$ to $\Q$ is the direction $\Bcal \Q  = 	\xi_1 \Q +\xi_2 \i\partial_{x}{\Q}$, which
is the steepest descent direction of the quantity $\Jcal_{\bm\xi}\sts{u}$
at $\Q$ along the curve $\{Q_{\omega+\lambda \xi_1, c+\lambda \xi_2}\}_{\lambda\in \R}$.  At the same time, we have the algebraic relations
    \begin{align*}\Scal''_{\omega,c}\sts{\Q} \widetilde{\varphi}_{\omega, c} =  -\mathcal{B} \Q, \text{~~and ~~}
 \Scal'''_{\omega,c}\sts{\Q}\sts{ \widetilde{\varphi}_{\omega, c} , \widetilde{\varphi}_{\omega, c} , \widetilde{\varphi}_{\omega, c} } + 3 \sts{\mathcal{B}\widetilde{\varphi}_{\omega, c} , \widetilde{\varphi}_{\omega, c} }= \mathbf{d}_{\bm\xi}^{\tprime},
\end{align*}
and the following non-degenerate result
\begin{align}
     \Scal^{\dprime}_{\omega, c}\sts{\Q}
    \sts{\Bcal \Q, \Bcal \Q}
    =
    \left<\mathcal{L}\Bcal \Q, \Bcal \Q \right>
    \not = 0. \label{dir_bq}
    \end{align}
By \eqref{dir_xi}, \eqref{dir_bq} and the degeneracy of $d''\sts{\omega, c}$ for the case $\sigma\in \sts{1,~2}$ and $c=2z_0\sqrt{\omega}$,  we need choose $\widetilde{\phi}_{\omega, c}$ as the primary perturbation direction and $\B{\Q}$ as the secondary perturbation direction instead of the independence between them, that is, we should take the following approximation \begin{align}\label{decomp_v1}
\Q
    +
    {\lambda}\widetilde{\phi}_{\omega, c}
    +
    \rho({ \lambda})\B{\Q},
    \end{align}
up to the spatial translation and phase rotation invariances, where $\rho\sts{\lambda}$ can be ensured by restriction of the solution on the level set $\Jcal\sts{\Q}$ and indeed can be determined in \eqref{decomp_v2} (also see Lemma \ref{lem:d1}).

Note that
 \begin{align*}
 \widetilde{\varphi}_{\omega, c} \not\bot \text{Null}\sts{\mathcal{L}},
 \end{align*}
 this makes us to renormalize the tangent vector $\widetilde{\varphi}_{\omega, c}$
 to make the approximation decomposition \eqref{decomp_v2} realizable (see Lemma \ref{lem:d2}).
    The renormalization $\varphi_{\omega, c}$ of the tangent vector
    $\widetilde{\varphi}_{\omega, c}$  means the projection of the tangent vector
    $\widetilde{\varphi}_{\omega, c}$ on the co-dimension subspaces of
    $\text{Null}\sts{\mathcal{L}}\oplus \text{span}\sts{\mathcal{B}\Q}$.
    This renormalization preserves the degeneracy of the action functional
    $\Scal_{\omega, c}$ along the direction $\varphi_{\omega, c}$, that is
        \begin{align*}
     \Scal^{\dprime}_{\omega, c}\sts{\Q}
    \sts{\phi_{\omega, c}, \phi_{\omega, c}} = 0,
    \end{align*}
    and we also have the following expressions
    \begin{equation*}
  \Scal_{\omega, c}\sts{\Q + {\lambda}\phi_{\omega, c}
        +
        \rho({ {\lambda}})\B{\Q}
    }
  =
  \Scal_{\omega, c}\sts{\Q} +\frac{1}{6}\mathbf{d}_{\bm\xi}^{\tprime}\cdot \lambda^3
  +\so{\abs{\lambda}^3},
\end{equation*}
\begin{equation*}
  \Scal_{\omega, c}\sts{\Q + {\lambda}\phi_{\omega, c}
        +
        \rho({ {\lambda}})\B{\Q}+\eps
    }
  =
  \Scal_{\omega, c}\sts{\Q} +\frac{1}{6}\mathbf{d}_{\bm\xi}^{\tprime}\cdot \lambda^3
  +\Scal^{\dprime}_{\omega, c}\sts{\Q}\sts{\eps,~\eps}
  +\so{\abs{\lambda}^3+\norm{\eps}_{H^1}^2},
\end{equation*}
which means that if the radiation term $\eps$ can be ignored,
$\Scal_{\omega, c}$ is a local monotone function as $\lambda$
under the special perturbation ${\lambda}\phi_{\omega, c}
        +
        \rho({ {\lambda}})\B{\Q}$ near $\Q$ ,
that is to say, the perturbation in the direction $\phi_{\omega, c}$
can play the dominant role under this special perturbation.
This definite property helps us to show the orbital instability of
the solitary waves of \eqref{gdnls} with the Virial argument in the degenerate case.

\item The remainder $\varepsilon$ in \eqref{decomp_v2} is not only small,
but also has some orthogonal structures, which makes the action functional
$\Scal_{\omega, c}$ to possess almost coercivity (or convex) (see Lemma \ref{lem:coer}).
\end{enumerate}

Secondly, in order to show the orbital instability of the solitary waves
$\Q\sts{x-ct}\e^{\i\omega t}$ of \eqref{gdnls} for the degenerate case
$\sigma \in \sts{1,~2}$ and $c=2z_0\sqrt{\omega}$,
we also turn to the effective monotonicity formula.
Since the quadratic term in $\lambda$ of
\begin{align*}
\action{\i\sts{\Q + {\lambda}~\phi_{\omega, c}
        +
        \rho({ {\lambda}})~\B{\Q}
        +
        \eps}_{x}
}{\varphi_{\omega, c}},
\end{align*}
 which corresponds to the term in \eqref{Iyt}, has the indefinite sign. By introducing the perturbation of $\varphi_{\omega, c}$ in the subspace $\text{Null}\sts{\mathcal{L}}$ to obtain the cancelation effect in the quadratic term of \eqref{Iyt} in $\lambda$, we can construct the refined Virial quantity
\begin{equation}
  \Iscr\sts{t} =\action{\i\eps\sts{t}}{ \phi_{\omega, c}\sts{ x} + \alpha\lambda\sts{t } \Q\sts{ x} + \beta\lambda\sts{ t}\i \partial_x\Q\sts{ x} },
\end{equation}
 which has the monotone property in some sense (see \eqref{eIt}), to show the orbital instability of the solitary wave $\Q\sts{x-ct}\e^{\i\omega t}$ of \eqref{gdnls}.

At last, the paper is organized as following.  In Section
\ref{sect:Prel},  we show the modulation decomposition of the functions in the tube
$\Ucal\sts{Q, \delta}$,  and the coercivity property of the linearized operator
$\Scal_{\omega, c}''\sts{\Q}$ on the subspace with finite co-dimension;
In Section \ref{sect:eps dyn}, we deduce the equation obeyed by the radiation term
$\eps\sts{t,x}$, and show the dynamical estimates of the parameters
$\lambda\sts{t}$, $y\sts{t}$ and $\gamma\sts{t}$ by the geometric structures of
the radiation term. In Section \ref{sect:main proof},
we first construct the solutions of \eqref{gdnls} near the solitary wave with
the refined geometric structures,
then show the orbital instability of the solitary wave of \eqref{gdnls}
in the degenerate case by the dynamical behaviors of the radiation and the parameters,
and the Virial argument.

In Appendix \ref{app:dds}, we prove that the action functional $\Scal_{\omega, c}$
is indeed of class $\mathcal{C}^3_{loc}\sts{H^1, ~\R }$ at $\Q$
because of the positivity of $\abs{\Q}$.

%
%
%

\section{Preliminaries}\label{sect:Prel}
In this section, we make some preparations to study the
orbital instability of solitary waves of \eqref{gdnls}
for the degenerate case $\sigma \in \sts{1,~2}$ and $c=2z_0\sqrt{\omega}$.
For $u\in H^1\sts{\R}$, we define the action functional $\Scal_{\omega,c}$
by
\begin{equation}
\label{S}
\Scal_{\omega,c}\sts{u}=\Ecal\sts{u}+\omega~\Mcal\sts{u}+c~\Pcal\sts{u},
\end{equation}
then it follows by the definitions of the energy, mass and momentum that $\Scal_{\omega,c}\in \Ccal^2\sts{H^1\sts{\R}; \R}$. By the variational characterization of $\Q$, we have $\Scal'_{\omega, c}\sts{\Q} =0$. For the convenience, we denote
$\Scal_{\omega, c}\sts{u} = \Qcal\sts{u}-\Ncal\sts{u}$ with
\begin{equation}\label{QN}
  \Qcal\sts{u}
  =
    \int\sts{   \frac{1}{2} \abs{{u}_x}^2
    +
    \frac{\omega}{2}~\abs{u}^2
    + \frac{c}{2}~\Re\sts{\i\bar{u}{u}_x} } \d x , \quad   \Ncal\sts{u}
  =
  \frac{1}{2\sigma + 2}
  \Re\int {\i\abs{u}^{2\sigma}\bar{u}{u}_x}\d x,
\end{equation}
and define
\begin{equation}
\label{d}
d\sts{\omega,c}= \Scal_{\omega,c}\sts{\Q}.
\end{equation}

\subsection{Basic properties of $\Scal_{\omega,c}$ and $d\sts{\omega,c}$}

By the definition of $\Scal_{\omega,c}$ in \eqref{S},
we know that $\Scal_{\omega,c}\in \Ccal^3\sts{H^1\sts{\R},~\R}$ for $\sigma\geq\frac{3}{2}$,
and  $\Scal_{\omega,c}\in \Ccal^2\sts{H^1\sts{\R},~\R}$  for $1<\sigma<\frac{3}{2}$.
In fact,  by the straightforward inspections, we can get the following identities
\begin{equation*}
\action{
  \Scal_{\omega,c}^{\prime}\sts{u}}{h}
  =
  \Re\int
  \sts{
    {u}_x\bar{h}_x + \omega~{u}\bar{h}
    + c~\i {u}_x\bar{h}
    - \i\abs{u}^{2\sigma}u_x\bar{h}
  }\d x ,
\end{equation*}
and
\begin{multline}\label{Sdp}
  \Scal_{\omega,c}^{\dprime}\sts{u}\sts{h,g}
  =
  \Re\int
  \sts{
   h_x\bar{g}_x + \omega~ h\bar{g}+c\i h_x\bar{g}
  }\d x
\\
   -
   \Re\int
   \sts{
   \i\abs{u}^{2\sigma}h_x\bar{g}
   +
    \i\sigma\abs{u}^{2\sigma-2}\bar{u}u_x h\bar{g}
    +
    \i\sigma\abs{u}^{2\sigma-2}{u}{u}_x \bar{h}\bar{g}
   }\d x,
\end{multline}
where $u\in H^1\sts{\R}$ and $g$, $h\in H^1\sts{\R}$. By means of the symmetries of \eqref{gdnls} and ODE theory, we have the explicit characterization of the kernel of the linearized operator $\Scal''_{\omega, c}\sts{\Q}$.
\begin{lemm}
The null space of the linearized operator $\mathcal{L}=\Scal''_{\omega, c}\sts{\Q}$ is characterized by
\begin{align*}
\text{Null}\sts{\mathcal{L}} = \text{span}\{\i \Q, ~ \partial_x \Q\}.
\end{align*}
\end{lemm}
\begin{proof}
Please refer to Proposition $3.6$ in \cite{LSS2013}.
\end{proof}
In addition, although  $\Scal_{\omega,c}\not\in\Ccal^3\sts{H^1\sts{\R},~\R}$ for $1<\sigma<\frac{3}{2}$ in general, we can still obtain the following local smoothing result of $\Scal_{\omega,c}$ near $\Q$ because of the positivity of $\abs{\Q}$. Its proof is straightforward and presented in Appendix \ref{app:dds}.
\begin{lemm}
\label{lem:S3d} Let $\sigma \in (1, 2)$ and $c=2z_0\sqrt{\omega}$, then the functional $\Scal_{\omega,c}$ is  of class $\mathcal{C}^3$ at $\Q$, and we have for any $f$, $g$ and $h\in H^1\sts{\R}$,
\begin{align*}
  \Scal_{\omega,c}^{\dprime}\sts{\Q+f}\sts{h,g}
  -
  \Scal_{\omega,c}^{\dprime}\sts{\Q}\sts{h,g}
  -
  \Scal_{\omega,c}^{\tprime}\sts{\Q}\sts{f,h,g}=\so{ \norm{f}_{H^1} },
\end{align*}
as $\norm{f}_{H^1}$ goes to zero. In fact, for any $~f,~g,~h\in H^1$, we have
\begin{align*}
  & \; \Scal_{\omega,c}^{\tprime}\sts{\Q}\sts{f,h,g}
  \\
  = &\;
  \Ncal^{\tprime}\sts{\Q}\sts{f,h,g}
\\
  = &
  -
  \Re\int
  \i\sigma\abs{\Q}^{2\sigma-2}
  \sts{
    {\Q}\bar{f}
    +
    \bar{Q}_{\omega,c}{f}
  }{h}_x\bar{g} \; dx \nonumber
  \\
  & -
  \Re\int
  \i\sigma
    \abs{{\Q}}^{2\sigma-2}
  \sts
  {
     {\bar{Q}_{\omega,c}}{f}_x
  +
    \sigma
    \partial_{x}{\Q}\bar{f}
    +
    \sts{\sigma-1}
    \frac{{\bar{Q}_{\omega,c}}^2}{\abs{\Q}^2} \partial_{x}{\Q} {f}
  }h\bar{g}  \; dx  \nonumber
\\
& -
  \Re\int
  \i\sigma
    \abs{{\Q}}^{2\sigma-2}
  \sts
  {
    {{\Q}}{f}_x
  +
    \sigma
    \partial_{x}{\Q} f
    +
    \sts{\sigma-1}
    \frac{{Q}_{\omega,c}^2}{\abs{\Q}^{2}}
     \partial_{x}{\Q} \bar{f}
  }\bar{h}\bar{g} \; dx.  \numberthis \label{S3dQ}
\end{align*}
\end{lemm}
\begin{rema}From the proof of Lemma \ref{lem:S3d},  we have the following identity
\begin{equation}\label{S3dsym}
  \Scal_{\omega,c}^{\tprime}\sts{\Q}\sts{f_1,f_2,f_3}
  =
  \Scal_{\omega,c}^{\tprime}\sts{\Q}
  \sts{f_{\tau_1},f_{\tau_2},f_{\tau_3}},
\end{equation}
where $f_1$, $f_2$, $f_3 \in H^1\sts{\R}$ and $\ltl{\tau_1,\tau_2,\tau_3}$ is the permutation of the set $\ltl{1,2,3}$. Hence, for any $f$, $g\in H^1\sts{\R}$, we have
\begin{multline*}
  \Scal_{\omega,c}^{\tprime}\sts{\Q}\sts{f,f,g}
  =
  -2\sigma\Re\int\i\abs{\Q}^{2\sigma-2}\sts{{\Q}\bar{f} {f}_x
  + \bar{Q}_{\omega,c}{f} {f}_x
  +
    \sigma
    \partial_{x}{\Q} \bar{f} f}\bar{g}\;
  \\
  -
  \sigma\sts{\sigma-1}\Re\int
  \i
    \abs{\Q}^{2\sigma-4}
  \sts
  {
    \bar{Q}_{\omega,c}^2 \partial_{x}{\Q} {f}f
    +
    {Q}_{\omega,c}^2 \partial_{x}{\Q} \bar{f}\bar{f}
  }\bar{g}
  \numberthis\label{S3dffg}
\end{multline*}
In addition, we have the Taylor series expression by Lemma
\ref{lem:S3d},
\begin{multline}
\label{TylS}
  \Scal_{\omega,c}\sts{\Q+h}
  =
  \Scal_{\omega,c}\sts{\Q}
    +
  \frac{1}{2}\Scal_{\omega,c}^{\dprime}\sts{\Q}\sts{h,h}
  +
  \frac{1}{6}\Scal_{\omega,c}^{\tprime}\sts{\Q}\sts{h,h,h}
  +
  \so{\norm{h}_{H^1}^3},
\end{multline}
as $\norm{h}_{H^1}$ goes to zero, where we used the fact that $\Scal'_{\omega, c}\sts{\Q} =0$.
\end{rema}

Now we turn to investigate the properties of the function $d\sts{\omega,c} = \Scal_{\omega,c}\sts{\Q}$ for $4\omega > c^2$. By the definition of $\Scal_{\omega,c}$ and the equation \eqref{Qeq}, we have
\begin{equation}
\label{dp}
d'\sts{\omega,c}
=
\begin{pmatrix}
	\partial_{\omega}{d\sts{\omega,c}} & \partial_{c}{d\sts{\omega,c}}
\end{pmatrix},
\end{equation}
and
\begin{equation}
\label{ddp}
d''\sts{\omega,c}
=
\begin{bmatrix}
	\partial_{\omega}^{2}{d\sts{\omega,c}} & \partial^2_{\omega,c}{d\sts{\omega,c}}
	\\[6pt]
	\partial^2_{c,\omega}{d\sts{\omega,c}} & \partial_{c}^{2}{d\sts{\omega,c}}
\end{bmatrix},
\end{equation}
where
\begin{equation*}
  \partial_{\omega}{d\sts{\omega,c}}
  =
  \Mcal\sts{Q_{\omega,c}},
  \qquad
  \partial_{c}{d\sts{\omega,c}}
  =
  \Pcal\sts{Q_{\omega,c}}.
\end{equation*}
Since $\det d''\sts{\omega,c}=0$ for $c=2z_0\sqrt{\omega}$, there exists a vector ${\bm \xi}=\sts{\xi_1,\xi_2}^{\text{T}}\in\R^2\backslash \{\bm 0\}$ such that
\begin{equation}
\label{dg}
  d''\sts{\omega,c}\cdot {\bm \xi}={\mathbf 0},
\end{equation}
which means that the vector ${\bm\xi}$ is an eigenvector of the Hessian matrix $ d''\sts{\omega,c}$ corresponding to the zero eigenvalue. Consequently, by \eqref{ddp} and \eqref{dg}, we have the following orthogonal relations
\begin{align*}
  \action{\Q}{\tphi_{\omega,c}}
  =
  &
  \left.
  \frac{\d}{\d \lambda}
  \Mcal\sts{Q_{\omega+\lambda\xi_1,c+\lambda\xi_2}}
  \right|_{\lambda=0}
  \\
  =
  &
  \left.
  \frac{\d}{\d \lambda}
  \partial_{\omega}{d\sts{\omega+\lambda\xi_1,c+\lambda\xi_2}}
  \right|_{\lambda=0}
  =0,\numberthis \label{tdg1}
\end{align*}
and
\begin{align*}
  \action{\i \partial_{x}{\Q}}{\tphi_{\omega,c}}
  =
  &
  \left.
  \frac{\d}{\d \lambda}
  \Pcal\sts{Q_{\omega+\lambda\xi_1,c+\lambda\xi_2}}
  \right|_{\lambda=0}
  \\
  =
  &
  \left.
  \frac{\d}{\d \lambda}
  \partial_{c}{d\sts{\omega+\lambda\xi_1,c+\lambda\xi_2}}
  \right|_{\lambda=0}
  =0,\numberthis \label{tdg2}
\end{align*}
where
\begin{equation}
\label{tphi}
\tphi_{\omega, c}\sts{x}
\triangleq \left.\frac{\d}{\d\lambda}Q_{\omega+ \lambda \xi_1, c + \lambda \xi_2}\sts{x}\right|_{\lambda=0}
=
\xi_1 \partial_{\omega}{\Q}\sts{x}+\xi_2 \partial_{c}{\Q}\sts{x}.
\end{equation}

Let us define the functional
\begin{align}\label{J}
 \Jcal_{\bm\xi}\sts{u}
	=\xi_1\Mcal\sts{u}+\xi_2\Pcal\sts{u}
\end{align} and its derivative
\begin{equation*}
	\B{u} =  \Jcal'_{\bm \xi} \sts{u} = \xi_1 u +\xi_2 \i\partial_{x}{u}.
\end{equation*}
Based on the above definitions, we have the following identity
\begin{align*}
\Jcal_{\bm \xi}\sts{u} = \frac12 \action{\B u}{u}.
\end{align*}
 By differentiating \eqref{Qeq} with respect to $\omega$ and $c$ and integrating by parts, we also have the following algebraic identity
\begin{equation}\label{exptddp}
  \Scal_{\omega,c}^{\dprime}\sts{\Q}
  \sts{\tphi_{\omega,c},\psi}
  = - \action{\B{\Q}}{\psi},\quad \text{for any~~}
  \psi\in H^1\sts{\R},
\end{equation}
where $\Scal_{\omega,c}^{\dprime}$ is the second derivative of $\Scal_{\omega, c}$ and is defined by \eqref{Sdp}.

Because of the degenerate property of the Hessian matrix $d''\sts{\omega , c}$ for $c=2z_0\sqrt{\omega}$ in \eqref{dg}, we need explore the third order derivatives of $d\sts{\omega,c}$ with respect to $\omega$ and $c$. In fact, by straightforward calculations,  we have
\begin{lemm}Let $\sigma\in \sts{1,~2}$, $c=2z_0\sqrt{\omega}$, and ${\bm\xi}$ be a zero eigenvector of the Hessian matrix $ d''\sts{\omega,c}$ as that in \eqref{dg}, then we have
 \begin{equation}
\label{ndeg}
  \mathbf{d}_{\bm\xi}^{\tprime}
  \triangleq
  \left.
  \frac{\d^3}{\d \lambda^3}
  d\sts{\omega+\lambda\xi_1,~c+\lambda\xi_2}
  \right|_{\lambda=0}
  \neq
  0,
\end{equation}
where $\mathbf{d}_{\bm\xi}^{\tprime}$ can be expressed by
\begin{equation}
\label{exptdtp}
  \mathbf{d}_{\bm\xi}^{\tprime}
  =
  \Scal_{\omega,c}^{\tprime}\sts{\Q}
  \sts{
  \tphi_{\omega,c},~
  \tphi_{\omega,c},~
  \tphi_{\omega,c} }
  +
  3
  \action{\B{\tphi_{\omega,c}}}{\tphi_{\omega,c}}.
\end{equation}
\end{lemm}
\begin{proof}
See Lemma $1$ in \cite{Fukaya2017}.
\end{proof}
Without loss of generality, we will assume that
\begin{equation}\label{pd3}
  \mathbf{d}_{\bm\xi}^{\tprime}<0
\end{equation}
in the context. Otherwise, it holds by reversing $\bm\xi$ by $-\bm\xi$. In addition, we will drop the subscript with respect to $\omega$, $c$ and $\bm \xi$ if without confusion in the rest of the paper.

\subsection{Geometric decomposition of $u$ and landscape of $\Scal$ near $Q$} As noted in the introduction,  the higher order approximation of the solution $u$ of  \eqref{gdnls} to $Q$ should be taken for the degenerate case $\sigma\in \sts{1,~2}$ and $c=2z_0\sqrt{\omega}$. We firstly renormalize the tangent vector $\widetilde{\varphi}=\widetilde{\varphi}_{\omega, c}$ of the curve $\{Q_{\omega+\lambda \xi_1, c+\lambda \xi_2}\}_{\lambda\in \R}$ at $\Q$ since the structure
    $$\widetilde{ \varphi}_{\omega, c} \; \bot\; \text{Null}(\mathcal{L}) $$
    doesn't hold. This renormalization  means to project  $\widetilde{\varphi}_{\omega, c}$ on $\text{Null}\sts{\mathcal{L}}^{\bot}\cap \text{span}\{\mathcal{B}Q\}^{\bot}$. Now let
\begin{align}
\label{vphi}
  \varphi
=
  {\tphi}
  -
  a
  {\partial_x Q}
  -
  b
  {\i{Q}}
\end{align}
where $a$ and $b$ are defined by \eqref{quan:ab}, that is, they satisfy
  \begin{equation*}
    a \action{\i {Q}}  {\partial_x Q}
    +
    b \action{\i{Q}}{\i{Q}}
    =  \action{\i {Q}}{\tphi},
  \end{equation*}
  and
  \begin{equation*}
    a \action{{Q}_{x}}  {\partial_x Q}
    +
    b \action{{Q}_{x}}{\i{Q}}
    =
    \action{{Q}_{x}}{\tphi},
  \end{equation*}
then it is easy to see that
\begin{equation}\label{pvtQ}
  \phi\;\bot\;\i Q
  \quad
  \text{and}
  \quad
  \phi\; \bot \;\partial_x Q.
\end{equation}
Secondly, by \eqref{tdg1} and \eqref{tdg2}, we have
\begin{equation}  \label{dgf}
  \action{\phi}{Q}=0
  \quad\text{and}\quad
  \action{\phi}{\i \partial_x Q}=0.
\end{equation}
By \eqref{exptddp}, \eqref{exptdtp} and $\text{Null}\sts{\mathcal{L}}=\text{span}\{\i Q, \partial_x Q\}$, we get the following identities
\begin{equation}\label{expddp}
  \Scal^{\dprime}\sts{Q}\sts{\phi,\psi}
  = - \action{\B{Q}}{\psi},\quad \text{for any~~}
  \psi\in H^1,
\end{equation}
and
\begin{equation}
\label{expdtp}
  \Scal^{\tprime}\sts{Q}
  \sts{
  \phi,~
  \phi,~
  \phi }
  +
  3
  \action{\B{\phi}}{\phi} = \mathbf{d}_{\bm\xi}^{\tprime}
 \not = 0.
\end{equation}
(See Proposition $1$ in \cite{Fukaya2017}).

After this, we introduce a lemma, which means that we can consider  $\varphi$ as the primary perturbation direction  and $\Bcal Q$ as the secondary perturbation direction instead of the independence between them.
\begin{lemm}
\label{lem:d1} There exist $0< \bar{\lambda}_0\ll 1$ and a $\mathcal{C}^1$ function $\widetilde{\rho}:\sts{-\bar{\lambda}_0~,~\bar{\lambda}_0}\mapsto\R$ such that if $\lambda\in\sts{-\bar{\lambda}_0~,~\bar{\lambda}_0}$, then we have \begin{equation*}
  \Jcal\sts{Q+\lambda\phi + \widetilde{\rho}\sts{\lambda}\B{Q}}=\Jcal\sts{Q},
\end{equation*}
where
\begin{equation}
\label{tp12}
  \widetilde{\rho}\sts{\lambda} = -\frac{\action{\B{\phi}}{\phi} }{2\action{\B{Q}}{\B{Q}}}\lambda^2 +\so{\lambda^2},
\quad
  \frac{\d}{\d\lambda}\widetilde{\rho}\sts{\lambda} = -\frac{\action{\B{\phi}}{\phi} }{\action{\B{Q}}{\B{Q}}}\lambda +\so{\abs{\lambda}}.
\end{equation}
%
\end{lemm}
\begin{proof}
It is the consequence of the Implicit Function Theorem, and please see Lemma $3$ in \cite{Fukaya2017}. We show the proof here for the convenience. In fact, we define the function
\begin{equation*}
  G\sts{\lambda, {\rho}}
  :=
  \Jcal\sts{Q+\lambda\,\phi +  {\rho}\,\B{Q}}
  -
  \Jcal\sts{Q}.
\end{equation*}
Firstly, it is easy to see that
$$G\sts{0,0}=0.$$ Secondly, the straightforward calculations imply that
\begin{equation*}
  \left.
  \frac{\partial}{\partial {\rho}}
  G\sts{\lambda,{\rho}}
  \right|_{\sts{\lambda, {\rho}}=\sts{0,0}}
  =\action{\B{Q}}{\B{Q}}\neq 0.
\end{equation*}
The Implicit Function Theorem implies that there exists
$0<\bar{\lambda}_0\ll 1$ and a $\Ccal^{\infty}$ function $\widetilde{\rho}:\sts{-\bar{\lambda}_0, \bar{\lambda}_0}\to\R$ such that
\begin{equation*}
  g\sts{\lambda}:=G\sts{\lambda,\widetilde{\rho}\sts{\lambda}}\equiv 0,
  \quad\text{for all }
  \lambda\in\sts{-\bar{\lambda}_0, \bar{\lambda}_0}.
\end{equation*}
Thirdly, by differentiating the function
$g\sts{\lambda}$ with respect to
$\lambda$ at point $0$, we have
\begin{equation*}
 \left.\frac{\d}{\d\lambda}g\sts{\lambda}\right|_{\lambda=0}
 =0\quad
 \text{and}\quad
 \left.\frac{\d^2}{\d\lambda^2}g\sts{\lambda}\right|_{\lambda=0}
 =0.
\end{equation*}
On one hand, we have
\begin{equation*}
  \action{\B{Q}}{\phi}
  +
  \action{\B{Q}}{\B{Q}}
  \left.
  \frac{\d}{\d\lambda}\widetilde{\rho}\sts{\lambda}
  \right|_{\lambda=0}=0,
\end{equation*}
This together with \eqref{dgf} implies that
\begin{equation*}
  \left.
  \frac{\d}{\d\lambda}\widetilde{\rho}\sts{\lambda}
  \right|_{\lambda=0}=0.
\end{equation*}
On the other hand, we have
 \begin{equation*}
 0
 =
 \left.\frac{\d^2}{\d\lambda^2}g\sts{\lambda}\right|_{\lambda=0}
 =
 \left.
    \frac{\d^2}{\d\lambda^2}\widetilde{\rho}\sts{\lambda}
 \right|_{\lambda=0}
 \action{\B{Q}}{\B{Q}}+\action{\B{\phi}}{\phi}.
\end{equation*}
By the Fundamental Theorem of Calculus, it is easy to see that
\eqref{tp12} holds. This completes the proof of the lemma.
\end{proof}

From now on, we will take
\begin{equation}
\label{rho}
  \rho\sts{\lambda} = -\frac{\action{\B{\phi}}{\phi} }{2\action{\B{Q}}{\B{Q}}}\lambda^2,
\end{equation}
and for $\delta>0$, we define the $\delta$-tube $\Ucal\sts{Q,~\delta}$ near $Q$  as
\begin{equation}\label{tube}
  \Ucal\sts{Q~,~\delta} = \ltl{u\in H^1\sts{\R}~:~ \inf_{y\in\R,\gamma\in\R}\norm{u\sts{\cdot}- Q\sts{\cdot-y}\e^{\i\gamma} }_{H^1}<\delta }.
\end{equation}
By the Implicit Function Theorem, we have the following refined modulation decomposition of the function $u$ in the $\delta$-tube $\Ucal\sts{Q,~\delta}$.
\begin{lemm}
\label{lem:d2}
    There exist $0<\bar{\delta}_1\ll 1$ and a unique $\Ccal^1$ map
    $
        \sts{ y,~\gamma,~\lambda }
        :
        \Ucal\sts{Q~,~\bar{\delta}_1}
        \mapsto
        \R
        \times
        \R\times\R,
    $
    such that if $u\in  \Ucal\sts{Q~,~\overline{\delta}_1}$,   and $\eps_{y,\gamma,\lambda }\sts{x}$ is defined by
     \begin{equation*}
  \eps_{y,\gamma,\lambda}\sts{x}
  \triangleq
  u\sts{x+y}\e^{\i\gamma}
  -
  \bsts{Q+ \lambda\phi+\rho\sts{\lambda}\B{Q}}\sts{x},
\end{equation*}
where $\rho\sts{\lambda}$ is determined by  \eqref{rho}, then we have
    \begin{equation*}
      \eps_{y,\gamma,\lambda}
      \; \bot\; \i Q,
      \quad
      \eps_{y,\gamma,\lambda}
      \; \bot\; Q_x
      \quad
      \text{and}
      \quad
      \eps_{y,\gamma,\lambda}
      \; \bot\; \phi.
    \end{equation*}
    Moreover, there exists a constant $C> 0$ such that if $u\in  \Ucal\sts{Q~,~\delta}$ with $0<\delta<\overline{\delta}_1$, then
    \begin{equation*}
      \norm{\eps_{y,\gamma,\lambda}}_{H^1}
      +
      \abs{y}  +
      \abs{\gamma} +
      \abs{\lambda} \leq C \delta.
    \end{equation*}
\end{lemm}
\begin{proof}
Let us define the following vector-valued functional of
$\sts{u,~{y},~{\gamma},~{\lambda}}$:
  \begin{align*}
    \bm{F}\sts{u;~ {y}, {\gamma}, {\lambda}}
    &
    =
    \sts{
        {F}^1\sts{u;~ {y}, {\gamma}, {\lambda}},
        ~{F}^2\sts{u;~ {y}, {\gamma}, {\lambda}},
        ~{F}^3\sts{u;~ {y}, {\gamma}, {\lambda}}
    },
  \end{align*}
where
  \begin{align*}
    {F}^1\sts{u;~ {y}, {\gamma}, {\lambda}}
    =&\;
    \Re\int \eps_{ {y}, {\gamma}, {\lambda}}~\overline{Q_{x}},
    \\
    {F}^2\sts{u;~ {y}, {\gamma}, {\lambda}}
    =&\;
    \Re\int \eps_{ {y}, {\gamma}, {\lambda}}~\overline{\i Q},
    \\
    {F}^3\sts{u;~ {y}, {\gamma}, {\lambda}}
    =& \;
    \Re\int \eps_{ {y}, {\gamma}, {\lambda}}~\overline{\phi}.
  \end{align*}
It is easy to find that
  \begin{align*}
    \frac{\partial\eps_{ {y}, {\gamma}, {\lambda}}}{\partial y}\Big|_{\sts{y,~\gamma,~\lambda,~u}=\sts{0,~0,~0,~Q}}
    =& \;
    {Q}_{x},\\
      \frac{\partial\eps_{ {y}, {\gamma}, {\lambda}}}{\partial \gamma}\Big|_{\sts{y,~\gamma,~\lambda,~u}=\sts{0,~0,~0,~Q}}
    = & \;
    \i Q , \\
        \frac{\partial\eps_{ {y}, {\gamma}, {\lambda}}}{\partial \lambda}\Big|_{\sts{y,~\gamma,~\lambda,~u}=\sts{0,~0,~0,~Q}}
    = & \;
    -\phi.
  \end{align*}
Thus, the Jacobian matrix of the vector-valued function $\bm{F}$ at $\sts{0,~0,~0,~Q}$ is
  \begin{equation*}
\frac{\partial \bm{F}}{\partial \sts{y,\gamma,\lambda}  }
    \sts{0,0,0, Q}
    =
    \begin{bmatrix}
      \action{{Q}_{x}}{{Q}_{x}} &
      \action{{Q}_{x}}{\i{Q}} &
      -\action{{Q}_{x}}{\phi}
      \\[6pt]
      \action{\i {Q}}{{Q}_{x}} & \action{\i{Q}}{\i{Q}} &
      -\action{\i {Q}}{\phi}
      \\[6pt]
      \action{\phi}{{Q}_{x}} & \action{\phi}{\i{Q}} &
      -\action{\phi}{\phi}
    \end{bmatrix},
  \end{equation*}
which implies that it is non-degenerate by \eqref{pvtQ} since
  \begin{equation*}
    \det \frac{\partial \bm{F}}{\partial \sts{y,\gamma,\lambda}  }
    \sts{0,0,0, Q}
    =
    -
    \sts{
        \norm{Q_x}_{2}^2\norm{Q}_{2}^2
        -
        \action{{Q}_{x}}{\i{Q}}^2
    }
    \norm{\phi}_{2}^2
    \neq 0.
  \end{equation*}
We can obtain the result by the Implicit Function Theorem.
\end{proof}

\begin{rema}We often call the remainder $\eps_{ y, \gamma, \lambda}$ the radiation of $u$. From the above proof, we used the non-degenerate property of the matrix
$
 \partial \bm{F}/ \partial \sts{y,\gamma,\lambda}
$
at $\sts{0,0,0, Q}$, which is ensured by the fact that
$$\varphi \; \bot\; \text{Null}(\mathcal{L}).$$
 That is the reason why we need replace $\widetilde{\varphi}$ with its renormalization $\varphi$ in the approximation of $u$ to $Q$ up to the spatial translation and phase rotation invariances.
\end{rema}

From the above decomposition, the radiation term $\eps$ of $u$ has orthogonal relation with $\varphi$ if  $u$ is in the $\delta$-tube $\Ucal\sts{Q,~\delta}$. In fact, the interaction between the radiation term $\eps$ and  $\Bcal Q$ is more smaller  than $\norm{\eps}_{H^1}$ if $u$ has the same $\Jcal$ quantity with $Q$, that is,
\begin{lemm}\label{lem:eBQ}
There exist $0<\bar{\delta}_2\ll 1$ and $0<\bar{\lambda}_2\ll 1$ such that if $\abs{ {\lambda}}\leq\bar{\lambda}_2$ and $\eps \in H^1\sts{\R}$ with  $\norm{\eps}_{H^1}\leq\bar{\delta}_2$ satisfy
\begin{equation*}
  \Jcal\sts{Q + {\lambda}\phi
        +
        \rho({ {\lambda}})\B{Q}+\eps
    }=\Jcal\sts{Q},
\end{equation*}
where $\rho({ {\lambda}})$ is determined by \eqref{rho}, then we have
\begin{equation}\label{eeBQ}
  {\action{\eps}{\B{Q}}}
  =
  \bo{
    \abs{ {\lambda}} \norm{\eps}_{H^1}
    +
    \norm{\eps}_{H^1}^2
  }
    +
    \so{ {\lambda}^2}.
\end{equation}
\end{lemm}
\begin{proof}
By the definition of the functional $\Jcal$ in \eqref{J}, we have
%
  \begin{align*}
    0
    =
    &\;
    \Jcal
    \sts{
        Q + {\lambda}\phi
        +
        \rho({ {\lambda}})\B{Q}+\eps
    }-\Jcal\sts{Q}
    \\
    =
    &\;
    \action{\B{Q}}{  {\lambda}\phi
        +
        \rho({ {\lambda}})\B{Q}+\eps }
    +
    \Jcal\sts{  {\lambda}\phi
        +
        \rho({ {\lambda}})\B{Q}+\eps }
    \\
    =
    &\;
    \action{\B{Q}}{\eps}
    +\rho( {\lambda})\action{\B{Q}}{\B{Q}}
    \\
    &\;+ {\lambda}^2 \Jcal\sts{\phi}
    + {\lambda}\rho( {\lambda})\action{\B{Q}}{\B{\phi}}
    + {\lambda}\action{\B{\phi}}{\eps}
    +\rho( {\lambda})\action{\B{Q}}{\eps}
    +\rho( {\lambda})^2\Jcal\sts{\B{Q}}
    +\Jcal\sts{\eps},
  \end{align*}
where we used the fact that $\action{\phi}{\B{Q}}=0$ in third equality by \eqref{dgf}. This implies that
\begin{align*}
  \action{\eps}{\B{Q}}
  =
  &
    -\rho( {\lambda})\action{\B{Q}}{\B{Q}}
    - {\lambda}^2 \Jcal\sts{\phi}
    \\
    &  - {\lambda}\rho( {\lambda})\action{\B{Q}}{\B{\phi}}
    - {\lambda}\action{\B{\phi}}{\eps}
    -\rho( {\lambda})\action{\B{Q}}{\eps}
    - \rho( {\lambda})^2\; \Jcal\sts{\B Q}
    - \Jcal\sts{\eps}.
    \numberthis \label{eBQ}
\end{align*}
By the fact that $ \Jcal \sts{\phi} = \frac12\action{\B{\phi}}{\phi} $, \eqref{rho} and the Cauchy-Schwarz inequality, we have the following estimates
\begin{align*}
  \rho( {\lambda})\action{\B{Q}}{\B{Q}}
    +  {\lambda}^2\Jcal\sts{\phi}
  =& \;0, \\
  {\lambda}\rho( {\lambda})\action{\B{Q}}{\B{\phi}}
  = \;
  \bo{ \abs{\lambda}^3}, &  \quad   \rho( {\lambda})^2\;\Jcal\sts{\B Q}
  =  \;
  \bo{
    {\lambda}^4
  },
\\
 {\lambda}\action{\B{\phi}}{\eps}
= \;
\bo{
  \abs{ {\lambda}} \norm{\eps}_{H^1}
}, \quad
  \rho( {\lambda})\action{\B{Q}}{\eps}
  =& \;
  \bo{
    {\lambda}^2 \norm{\eps}_{H^1}
  }, \quad
 \Jcal\sts{\eps}
=  \;
\bo{
      \norm{\eps}^2_{H^1}
}.
\end{align*}
Inserting the above estimates into \eqref{eBQ}, 
we can obtain the result, and complete the proof of the lemma.
\end{proof}

The following result shows that the action functional $\Scal$ has definite dynamics at $Q$ along the special perturbation ${\lambda}\phi + \rho({ {\lambda}})\B{Q}$ although $d\sts{\omega, ~c}=\Scal_{\omega,c}\sts{\Q}$ has the degenerate Hessian matrix $d''\sts{\omega,~c}$ for $c=2z_0\sqrt{\omega}$.
\begin{lemm}
\label{lem:tyl1}
There exists $0<\bar{\lambda}_3\ll 1$ such that for any $\lambda\in\R$ satisfies $0<\abs{\lambda}<\bar{\lambda}_3$,
we have
\begin{equation*}
  \Scal\sts{Q + {\lambda}\phi
        +
        \rho({ {\lambda}})\B{Q}
    }
  =
  \Scal\sts{Q} +\frac{1}{6}\mathbf{d}_{\bm\xi}^{\tprime}\cdot \lambda^3
  +\so{\abs{\lambda}^3},
\end{equation*}
where $\rho\sts{\lambda}$ and $\mathbf{d}_{\bm\xi}^{\tprime}$ are defined by \eqref{rho} and \eqref{ndeg} respectively.
\end{lemm}
\begin{proof}
Since $\abs{\lambda}$ is small enough, it follows from \eqref{rho} that
$
  \norm{ {\lambda}\phi
            +
            \rho({ {\lambda}})\B{Q}}_{H^1}
$
is small enough. Hence we have by the Taylor series expression in \eqref{TylS} that
\begin{align}
        \Scal\sts{Q + {\lambda}\phi
        +
        \rho({ {\lambda}})\B{Q}
    }
    = & \;
    \Scal\sts{Q}
    +
    \frac{1}{2}\Scal^{\dprime}\sts{Q}
    \sts{ {\lambda}\phi
        +
        \rho({ {\lambda}})\B{Q},~
         {\lambda}\phi
        +
        \rho({ {\lambda}})\B{Q}}
        \nonumber
        \\
        &
        +
        \frac{1}{6}
        \Scal^{\tprime}\sts{Q}
        \sts{
         {\lambda}\phi
        +
        \rho({ {\lambda}})\B{Q},~
         {\lambda}\phi
        +
        \rho({ {\lambda}})\B{Q},~
         {\lambda}\phi
        +
        \rho({ {\lambda}})\B{Q}
        }
        \nonumber
        \\
        &
        +
        \so{ \norm{ {\lambda}\phi
            +
            \rho({ {\lambda}})\B{Q}}_{H^1}^3 }
        ,
        \label{ES1}
\end{align}
where we used the fact that $\Scal'\sts{Q}=0$ in the right hand side.

Firstly, by \eqref{dgf} and \eqref{expddp}, we have
\begin{equation*}
  \Scal^{\dprime}\sts{Q}
    \sts{ \lambda\phi,~\lambda\phi}
  =-\lambda^2\action{\B{Q}}{\phi}=0,
\end{equation*}
and by \eqref{dgf} and \eqref{rho}, we have
\begin{equation*}
  \Scal^{\dprime}\sts{Q}
    \sts{ \rho({ {\lambda}})\B{Q},~\lambda\phi}
    =
    -\lambda\rho({ {\lambda}})
    \action{\B{Q}}{\B{Q}}
    =
    \frac{\lambda^3}{2}\action{\B{\phi}}{\phi},
\end{equation*}
\begin{equation*}
  \Scal^{\dprime}\sts{Q}
    \sts{
        \rho({ {\lambda}})\B{Q},~
        \rho({ {\lambda}})\B{Q}
    }
  =
  \so{\abs{\lambda}^3}.
\end{equation*}
Therefore, we obtain
\begin{align*}
\frac{1}{2}\Scal^{\dprime}\sts{Q}
    \sts{ {\lambda}\phi
        +
        \rho({ {\lambda}})\B{Q},~
         {\lambda}\phi
        +
        \rho({ {\lambda}})\B{Q}}
    =&
    \frac{\lambda^3}{2}\action{\B{\phi}}{\phi}
    +
    \so{\abs{\lambda}^3}.\numberthis \label{eS2}
\end{align*}

Secondly, by \eqref{rho}, we have
\begin{equation*}
  \rho\sts{\lambda}=\so{\abs{\lambda}},
\end{equation*}
which implies that
\begin{equation}
\label{eS3}
\frac{1}{6}
        \Scal^{\tprime}\sts{Q}
        \sts{
         {\lambda}\phi
        +
        \rho({ {\lambda}})\B{Q},~
         {\lambda}\phi
        +
        \rho({ {\lambda}})\B{Q},~
         {\lambda}\phi
        +
        \rho({ {\lambda}})\B{Q}
        }=
  \frac{\lambda^3}{6}
  \Scal^{\tprime}\sts{Q}
     \sts{\phi,~\phi,~\phi}
  +
  \so{\abs{\lambda}^3}.
\end{equation}

Lastly, it is easy to see from \eqref{rho} that
\begin{equation}
\label{eSo}
 \norm{ {\lambda}\phi
            +
            \rho({ {\lambda}})\B{Q}}_{H^1}^3  = \bo{\abs{\lambda}^3}.
\end{equation}
Inserting \eqref{eS2}, \eqref{eS3} and \eqref{eSo} into \eqref{ES1}, we obtain
\begin{align*}
      \Scal\sts{Q + {\lambda}\phi
        +
        \rho({ {\lambda}})\B{Q}
    }
      =
  & \;  \Scal\sts{Q}+
  \frac{\lambda^3}{6}
  \Big({
    3\action{\B{\phi}}{\phi}
    +
    \Scal^{\tprime}\sts{Q}
        \sts{
         \phi,~\phi,~\phi
        }
  }\Big)
  +
  \so{\abs{\lambda}^3}
  \\
  =
  & \;  \Scal\sts{Q}+
  \frac{1}{6}\mathbf{d}_{\bm\xi}^{\tprime}\cdot \lambda^3
  +
  \so{\abs{\lambda}^3},
\end{align*}
where we used \eqref{expdtp} in the second equality. This concludes the proof.
\end{proof}

\begin{lemm}
\label{lem:tyl2}
There exist $0<\bar{\delta}_4\ll 1$ and $0<\bar{\lambda}_4\ll 1$ such that if $\eps\in H^1\sts{\R}$ with $\norm{\eps}_{H^1}\leq\bar{\delta}_4$
and $\lambda$ with $\abs{\lambda}\leq\bar{\lambda}_4$
satisfy
\begin{align}\label{EpBQ}
    \action{\eps}{\B{Q}}
    =
    \bo{    \abs{\lambda}\norm{\eps}_{H^1}
        +   \norm{\eps}_{H^1}^2
    }
    +\so{\lambda^2},
\end{align}
then we have
\begin{equation}
\label{tyl2}
  \Scal\sts{Q + {\lambda}\phi
        +
        \rho({ {\lambda}})\B{Q}+\eps
    }
  =
  \Scal\sts{Q} +\frac{1}{6}\mathbf{d}_{\bm\xi}^{\tprime}\cdot \lambda^3
  +\Scal^{\dprime}\sts{Q}\sts{\eps,~\eps}
  +\so{\abs{\lambda}^3+\norm{\eps}_{H^1}^2},
\end{equation}
where $\rho\sts{\lambda}$ and $\mathbf{d}_{\bm\xi}^{\tprime}$ are defined by \eqref{rho} and \eqref{ndeg} respectively.
\end{lemm}
\begin{proof}
By the Taylor series expression of $\Scal$ at $Q$,  we have from Lemma \ref{lem:tyl1} that
\begin{align*}
        \Scal\sts{Q + {\lambda}\phi
        +
        \rho({ {\lambda}})\B{Q}
        +
        \eps
    }
    =&\;
    \Scal\sts{Q} +
    \frac{1}{2}\Scal^{\dprime}\sts{Q}
    \sts{
        {\lambda}\phi
        +
        \rho({ {\lambda}})\B{Q}
        +
        \eps,~
        {\lambda}\phi
        +
        \rho({ {\lambda}})\B{Q}
        +
        \eps
    }
        \\
        &
        +
        \frac{1}{6}
        \Scal^{\tprime}\sts{Q}
        \sts{
             {\lambda}\phi
            +
            \rho({ {\lambda}})\B{Q}
            +
            \eps,~
              {\lambda}\phi
            +
            \rho({ {\lambda}})\B{Q}
            +
            \eps,~
            {\lambda}\phi
            +
            \rho({ {\lambda}})\B{Q}
            +
            \eps
        }
    \\
    &
        +
        \so{ \norm{ {\lambda}\phi
            +
            \rho({ {\lambda}})\B{Q}+\eps}_{H^1}^3 }
    \\
    =&\;
    \Scal\sts{Q} +
    \frac{1}{6}\mathbf{d}_{\bm\xi}^{\tprime}\cdot \lambda^3
    +
    \Scal^{\dprime}\sts{Q}
    \sts{
        {\lambda}\phi
            ,~
        \eps
    }
    +
    \frac{1}{2}
    \Scal^{\dprime}\sts{Q}
    \sts{\eps,~\eps}
    \\
    &
    +
    \bo{\lambda^2\norm{\eps}_{H^1} +\norm{\eps}_{H^1}^3 }
    +
    \so{
        \abs{\lambda}^3
        +
        \norm{\eps}_{H^1}^3
    }
    \\
    =&\;
    \Scal\sts{Q} +
    \frac{1}{6}\mathbf{d}_{\bm\xi}^{\tprime}\cdot \lambda^3
    -
    \lambda\action{\eps}{\B{Q}}
    +
    \frac{1}{2}
    \Scal^{\dprime}\sts{Q}
    \sts{\eps,\eps}
    \\
    &
    +
    \bo{\lambda^2\norm{\eps}_{H^1} +\norm{\eps}_{H^1}^3 }
    +
    \so{
        \abs{\lambda}^3
        +
        \norm{\eps}_{H^1}^2
    },\numberthis\label{Sle}
\end{align*}
where we used the facts that $\Scal'\sts{Q}=0$ in the first equality and that
\begin{align*}
 \abs{\Scal^{\dprime}\sts{Q}
\sts{
        \rho({ {\lambda}})\B{Q},~\eps
    }} + \abs{\Scal^{\tprime}\sts{Q}
        \sts{
             {\lambda}\phi
            +
            \rho({ {\lambda}})\B{Q}
            ,~
              {\lambda}\phi
            +
            \rho({ {\lambda}})\B{Q}
            ,~
                        \eps
        }} = \bo{\lambda^2\norm{\eps}_{H^1}}, \\
             \abs{ \Scal^{\tprime}\sts{Q}
        \sts{
             {\lambda}\phi
            +
            \rho({ {\lambda}})\B{Q}
            ,~
               \eps
            ,~
                        \eps
        }}= \bo{\abs{\lambda}\norm{\eps}^2_{H^1}}, \quad
             \abs{\Scal^{\tprime}\sts{Q}
        \sts{
             \eps
            ,~
               \eps
            ,~
                        \eps}
        }= \bo{\norm{\eps}^3_{H^1}}
\end{align*}
in the second equality.

By \eqref{EpBQ} and Cauchy-Schwarz's inequality, we have
\begin{equation}\label{tp3}
  \abs{\lambda \action{\eps}{\B{Q}} }
  =
  \bo{
    \lambda^2~\norm{\eps}_{H^1}
    +
    \abs{\lambda}~\norm{\eps}_{H^1}^2
  }
    +
    \so{\abs{\lambda}^3}
  =
  \so{\abs{\lambda}^3+\norm{\eps}_{H^1}^2},
\end{equation}
and
\begin{equation}
\label{tp4}
  \bo{\lambda^2\norm{\eps}_{H^1} +\norm{\eps}_{H^1}^3 }
  =
  \so{\abs{\lambda}^3+\norm{\eps}_{H^1}^2}.
\end{equation}
Inserting \eqref{tp3} and \eqref{tp4} into \eqref{Sle}, we obtain the result and complete the proof.
\end{proof}

\subsection{Properties of the linearized operator $\mathcal{L}$}
As shown in Lemma \ref{lem:tyl2}, we are left to show that the quadratic term $\Scal''\sts{\eps, ~\eps}$ has some coercivity (or convex) property under the condition that the radiation term $\eps$ has some geometric orthogonal structures. It is the task in this subsection and related to the spectral
properties of the linearized operators $\mathcal{L}=\Scal''$.  The spectral
properties of the linearized operator around the solitary waves play a crucial role in long time dynamics of the solutions near solitary waves in \cite{CoPe2003, GSS1987, GSS1990, Maeda2012, MartelM:Instab:gKdV, MMR:dya:gKdV, MartelMT:Stab:gKdV, MartelMT:Stab:NLS, NakSchlag:book, Ohta:Ins:JFA, Weinstein1985, Wein:stab:CPAM} and references therein.


Now by the variational characterization of $Q$ and standard argument in \cite{GSS1987, MiaoTX:Stab, Weinstein1985} we can exhibit the following coercive property of the linearized operator $ {\Scal}^{\dprime}\sts{Q}$ in the energy space.
\begin{lemm}
\label{lem:coer}
There exists a constant $\kappa>0$ such that if $\eps\in H^1\sts{\R}$ satisfies
    \begin{equation}
    \label{vt}
      \action{\eps}{\i{Q}}=0, \quad \action{\eps}{{Q}_{x}}=0, \text{~~and~~}
      \action{\eps}{\phi}=0,
    \end{equation}
    then we have
    \begin{equation}
    \label{tp2}
      {\Scal}^{\dprime}\sts{Q}\sts{\eps,~\eps}\geq \kappa\norm{\eps}_{H^1}^2
      -
      \frac{1}{\kappa}\action{\eps}{\B{Q}}^2.
    \end{equation}
\end{lemm}
\begin{proof}
See proof in \cite{Fukaya2017, MiaoTX:Stab, Weinstein1985}.
\end{proof}

Combining Lemma \ref{lem:coer} with Lemma \ref{lem:eBQ}, we have
\begin{coro}
\label{cor:coer}There exist $0<\bar{\delta}_5\ll 1$ and $0<\bar{\lambda}_5\ll 1$ such that if $\eps\in H^1\sts{\R}$ with $\norm{\eps}_{H^1}\leq\bar{\delta}_5 $, and $\lambda\in\R$ with
$\abs{ {\lambda}}\leq\bar{\lambda}_5 $ satisfy
\begin{equation*}
  \action{\eps}{\i{Q}}=0, \quad \action{\eps}{{Q}_{x}}=0, \text{~~and~~}
  \action{\eps}{\phi}=0,
\end{equation*}and
\begin{equation*}
\Jcal\sts{Q + {\lambda}\phi
        +
        \rho({ {\lambda}})\B{Q}+\eps
    }=\Jcal\sts{Q},
\end{equation*}
where $\rho\sts{\lambda}$ is determined  by \eqref{rho}, then we have
\begin{equation*}
  {\Scal}^{\dprime}\sts{Q}\sts{\eps,\eps}\geq \frac{\kappa}{2}\norm{\eps}_{H^1}^2
  +
  \so{ {\lambda}^4}.
\end{equation*}
\end{coro}
\begin{proof}
First, by \eqref{eeBQ}, we have
\begin{align*}
\action{\eps}{\B{Q}}^2
=
&
  \bsts{
    \bo{
        \abs{ {\lambda}}\norm{\eps}_{H^1}
        +
        \norm{\eps}_{H^1}^2
    }
    +
    \so{ {\lambda}^2}
  }^2
\\
=
&
\bo{
    \lambda^2\norm{\eps}_{H^1}^2
    +
    \norm{\eps}_{H^1}^4
}
+
\so{\lambda^4}
\\
=
&
\so{
    \norm{\eps}_{H^1}^2
}
+
\so{\lambda^4},\numberthis \label{tp1}
\end{align*}
where $\abs{\lambda}$ and $\norm{\eps}_{H^1}$ should be taken sufficiently small.

Next, by inserting \eqref{tp1} into \eqref{tp2} and taking  $\norm{\eps}_{H^1}$ sufficiently small, we can obtain
  \begin{align*}
    {\Scal}^{\dprime}\sts{Q}\sts{\eps,\eps}
    \geq
    &
     \kappa\norm{\eps}_{H^1}^2
      -
      \frac{1}{\kappa}\action{\eps}{\B{Q}}^2
    \\
    \geq
    &
    \kappa\norm{\eps}_{H^1}^2
    +
    \so{\norm{\eps}_{H^1}^2}
    +
    \so{\lambda^4}
    \\
    \geq
    &
    \frac{\kappa}{2}\norm{\eps}_{H^1}^2
    +
    \so{ {\lambda}^4}.
  \end{align*}
This concludes the proof.
\end{proof}

\section{The equation on $\eps$ variable and the dynamical estimates of the parameters}\label{sect:eps dyn}
In this section, we derive the equation satisfied by the radiation term
\begin{equation}
\label{eps}
  \eps\sts{t,x}=u\sts{t,x+y\sts{t}}\e^{\i\gamma\sts{t}}
  -
  \bsts{
    Q\sts{x}
    +
    {\lambda\sts{t}}\phi\sts{x}
    +
    \rho({ {\lambda}\sts{t}})\B{Q}\sts{x}
    },
\end{equation}
where $u$ is a solution of \eqref{gdnls} in the energy space $H^1\sts{\R}$, $\phi$ and $\rho$ are defined by \eqref{vphi} and \eqref{rho} respectively, $\lambda$, $y$, $\gamma$ are the $\Ccal^1$ functions with respect to $t$ which  will be determined later. For convenience, we denote
\begin{equation}
\label{F}
  f\sts{u}= \i\abs{u}^{2\sigma}{u}_{x},
\end{equation}
and
\begin{align*}
  \Rcal_1\sts{Q,\eta}
=
  \i\abs{Q }^{2\sigma}{\eta }_{x}
  +
  \i\sigma\abs{Q }^{2\sigma-2}\bar{Q}{Q}_x{\eta }
  +
  \i\sigma\abs{Q }^{2\sigma-2}{Q}{Q}_x \bar{\eta }.
  \numberthis\label{R1}
\end{align*}

Firstly,  we have the following result.
\begin{lemm} \label{lem:eps}
Let $u(t)\in \mathcal{C}\sts{[0, T),~ H^1\sts{\R} }$ be the solution of \eqref{gdnls} for some $T>0$, and $\eps(t,x)$ be defined  by \eqref{eps}, then we have
\begin{align*}
 \i{\eps}_t ~-~  &
  \Lcal
  \bsts{ {\lambda}\phi
        +
        \rho({ {\lambda}})\B{Q}
        +
        \eps
  } \\
  = &\;
  -\i \lambda_t
  \bsts{\phi
        +
        \dot{\rho}({\lambda})\B{Q}
  }
   +
    \i \sts{y_t-c}~\bsts{Q + {\lambda}\phi
        +
        \rho({ {\lambda}})\B{Q}
        +
        \eps}_{x}
        \\
        & \;
  -\sts{\gamma_t+\omega}
  \bsts{Q + {\lambda}\phi
        +
        \rho({ {\lambda}})\B{Q}
        +
        \eps}
  \\ &\;
  -
  \Rcal_2
  \bsts{Q,~{\lambda}\phi
        +
        \rho({ {\lambda}})\B{Q}
        +
        \eps
  }
  -
  \tilde{\Rcal}
  \bsts{Q,~{\lambda}\phi
        +
        \rho({ {\lambda}})\B{Q}
        +
        \eps
  }, \numberthis \label{epst}
\end{align*}
where the linear term $\Lcal {\eta} $, the quadratic term $ \Rcal_2\sts{Q,~\eta}$ and the higher order term $ \tilde{\Rcal}\sts{Q,~\eta}$ are defined by
\begin{align}
  \Lcal {\eta} = &
  -\eta_{xx}+\omega \eta + c\i\eta_x
  -
   \i\abs{Q}^{2\sigma}\eta_x
     -
   \i\sigma\abs{Q}^{2\sigma-2}\bar{Q}Q_x \eta
   -
   \i\sigma\abs{Q}^{2\sigma-2}{Q}{Q}_x \bar{\eta},
\label{L}
\\
  \Rcal_2& \sts{Q,~\eta}
\;=
  \sigma\i\abs{Q}^{2\sigma-2}\bar{Q}{\eta}{\eta}_x
  +
  \sigma\i\abs{Q}^{2\sigma-2} {Q}\bar{\eta}{\eta}_x  +
  \sigma^2 \i\abs{Q}^{2\sigma-2} {Q}_x{\eta}\bar{\eta}   \nonumber
  \\
  & \;
  \qquad \qquad +
  \frac{\sigma\sts{\sigma-1}}{2}\i\abs{Q}^{2\sigma-4}{\bar{Q}}^2{Q}_x {\eta}{\eta}
  +
  \frac{\sigma\sts{\sigma-1}}{2}\i\abs{Q}^{2\sigma-4}{{Q}}^2{Q}_x \bar{\eta}\bar{\eta},
\label{R2}
\\
 & \tilde{\Rcal}\sts{Q,~\eta}
=
  f\sts{ Q+\eta }
  -
  f\sts{ Q } - \Rcal_1\sts{Q,\eta}
  -
  \Rcal_2\sts{Q,\eta}.
\label{Rt}
\end{align}
\end{lemm}
\begin{proof}
First, let
\begin{equation}
\label{uv}
  v\sts{t,x} = u\sts{t,x+y\sts{t}}\e^{\i\gamma\sts{t}},
\end{equation}
then we have
  \begin{multline*}
    {v}_{t}\sts{t,x}
    =
    {u}_{t}\sts{t,x+y\sts{t}}\e^{\i\gamma\sts{t}}
    +
    y_t~{u}_{x}\sts{t,x+y\sts{t}}\e^{\i\gamma\sts{t}}
    +
    \i\gamma_t ~ u\sts{t,x+y\sts{t}}\e^{\i\gamma\sts{t}},
    \numberthis\label{ut}
  \end{multline*}
  and
    \begin{equation}
    \label{ux}
    {u}_{x}\sts{t,x+y\sts{t}}\e^{\i\gamma\sts{t}}
    =
    {v}_{x}\sts{t,x},
\quad
  {u}_{xx}\sts{t,x+y\sts{t}}\e^{\i\gamma\sts{t}}
  =
  {v}_{xx}\sts{t,x}.
  \end{equation}
By inserting \eqref{uv} and \eqref{ux} into \eqref{ut}, we have
  \begin{align*}
    {u}_{t}\sts{t,x+y\sts{t}}\e^{\i\gamma\sts{t}}
    =
    &
    {v}_{t}\sts{t,x}
    -
    y_t~{v}_{x}\sts{t,x}
    -
    \i\gamma_t ~ {v}\sts{t,x}.
  \end{align*}
Hence, it follows from \eqref{gdnls} that
\begin{align}
\label{veq}
  \i{v}_{t}
    -
    \i y_t~{v}_{x}
    +
    \gamma_t ~ {v}
    +
    {v}_{xx}
    +
    f\sts{v}=0.
\end{align}
Now, by $v\sts{t,x}=Q\sts{x}+\lambda\sts{t} \varphi\sts{x} + \rho\sts{\lambda\sts{t}}\mathcal{B}Q\sts{x} + \eps\sts{t,x}$ and \eqref{veq}, we have
\begin{align*}
  0
  =&
  \i{\eps}_{t}
  +
  \sts{Q + {\lambda}\phi
        +
        \rho({ {\lambda}})\B{Q}
        +
        \eps}_{xx}
  +
  f
  \sts{{Q + {\lambda}\phi
        +
        \rho({ {\lambda}})\B{Q}
        +
        \eps}
  }
  \\
  &
  +
  \i \lambda_t
  \sts{\phi
        +
        \dot{\rho}({\lambda})\B{Q}
  }
  -
    \i y_t~\sts{Q + {\lambda}\phi
        +
        \rho({ {\lambda}})\B{Q}
        +
        \eps}_{x}
  +\gamma_t
  \sts{Q + {\lambda}\phi
        +
        \rho({ {\lambda}})\B{Q}
        +
        \eps}
\\
=
&
  \i{\eps}_{t}
  +
  \sts{{\lambda}\phi
        +
        \rho({ {\lambda}})\B{Q}
        +
        \eps}_{xx}
  -
  \omega
  \sts{ {\lambda}\phi
        +
        \rho({ {\lambda}})\B{Q}
        +
        \eps}
  -
  \i c~\sts{{\lambda}\phi
        +
        \rho({ {\lambda}})\B{Q}
        +
        \eps}_{x}
  \\
  &
  +
  \i \lambda_t
  \sts{\phi
        +
        \dot{\rho}({\lambda})\B{Q}
  }
  -
    \i \sts{y_t-c}~\sts{Q + {\lambda}\phi
        +
        \rho({ {\lambda}})\B{Q}
        +
        \eps}_{x}
  \\
  &
  +\sts{\gamma_t+\omega}
  \sts{Q + {\lambda}\phi
        +
        \rho({ {\lambda}})\B{Q}
        +
        \eps}
  +
  f
  \sts{{Q + {\lambda}\phi
        +
        \rho({ {\lambda}})\B{Q}
        +
        \eps}
  }-f\sts{Q}
  \\
  &
  +
  Q_{xx}-\omega~Q-\i c~Q_{x}+f\sts{Q}
\\
=
&
  \i{\eps}_{t}
    +
  \sts{ {\lambda}\phi
        +
        \rho({ {\lambda}})\B{Q}
        +
        \eps}_{xx}
  -
  \omega
  \sts{ {\lambda}\phi
        +
        \rho({ {\lambda}})\B{Q}
        +
        \eps}
  -
  \i c~\sts{ {\lambda}\phi
        +
        \rho({ {\lambda}})\B{Q}
        +
        \eps}_{x}
  \\
  &
  +
    \Rcal_1\sts{Q,{\lambda}\phi
        +
        \rho({ {\lambda}})\B{Q}
        +
        \eps}
    +
  \i \lambda_t
  \sts{\phi
        +
        \dot{\rho}({\lambda})\B{Q}
  }
  \\
  &
  -
    \i \sts{y_t-c}~\sts{Q + {\lambda}\phi
        +
        \rho({ {\lambda}})\B{Q}
        +
        \eps}_{x}
  +\sts{\gamma_t+\omega}
  \sts{Q + {\lambda}\phi
        +
        \rho({ {\lambda}})\B{Q}
        +
        \eps}
  \\
  &
  +
  f
  \sts{{Q + {\lambda}\phi
        +
        \rho({ {\lambda}})\B{Q}
        +
        \eps}
  }
  -
  f\sts{Q}
  -
  \Rcal_1\sts{Q,{\lambda}\phi
        +
        \rho({ {\lambda}})\B{Q}
        +
        \eps}
\\
=
&
  \i{\eps}_{t}
    -
  \Lcal
  \sts{ {\lambda}\phi
        +
        \rho({ {\lambda}})\B{Q}
        +
        \eps
  } \\ &
  +
  \i \lambda_t
  \sts{\phi
        +
        \dot{\rho}({\lambda})\B{Q}
  }
     -
    \i \sts{y_t-c}\sts{Q + {\lambda}\phi
        +
        \rho({ {\lambda}})\B{Q}
        +
        \eps}_{x}
  +\sts{\gamma_t+\omega}
  \sts{Q + {\lambda}\phi
        +
        \rho({ {\lambda}})\B{Q}
        +
        \eps}
  \\
  &
  +
  \Rcal_2
  \sts{Q,{\lambda}\phi
        +
        \rho({ {\lambda}})\B{Q}
        +
        \eps
  }
  +
  \tilde{\Rcal}
  \sts{Q,{\lambda}\phi
        +
        \rho({ {\lambda}})\B{Q}
        +
        \eps
  },
\end{align*}
where we used the fact that $Q$ satisfies the equation \eqref{Qeq} in the third equality.
This concludes the proof.
\end{proof}


Secondly, by the  reserved geometric structures of the radiation term $\eps$, we can obtain the dynamical estimates of the modulation parameters $\lambda\sts{t},~y\sts{t}$ and $\gamma\sts{t}$.

\begin{lemm}
\label{lem:lyot}
 Suppose $T>0$. There exist $0<\bar{\delta}_6\ll 1$ and $0<\bar{\lambda}_6\ll 1$ such that if for all $t\in [0,T)$, $\eps\sts{t,~x} $ and $\lambda\sts{t}$ satisfy the equation \eqref{epst} and
\begin{equation}
\label{tp5}
  \action{\eps\sts{t}}{\i{Q}}=0, \quad \action{\eps\sts{t}}{{Q}_{x}}=0
  \text{~~and~~}
  \action{\eps\sts{t}}{\phi}=0,
\end{equation}
and
\begin{equation}
\label{tp6}
 \norm{\eps\sts{t}}_{H^1}\leq\bar{\delta}_6
 \text{~~and~~}
 \abs{\lambda\sts{t}}\leq\bar{\lambda}_6,
\end{equation}
then we have
\begin{equation}
\label{lyot}
  \abs{ \lambda_t }
  +
  \abs{ y_t-c }
  +
  \abs{ \gamma_t+\omega }
  \leq
  C\sts{Q}\bsts{ \abs{\lambda}+\norm{\eps\sts{t}}_{H^1} }
\end{equation}
for all $ t\in [0,T)$, where $C\sts{Q}$ is a uniform constant in time $t$ which only depends on $Q$.
\end{lemm}
\begin{proof}%

First, for any $\psi\in H^1$, it follows from \eqref{Sdp} and \eqref{L} that,
\begin{align*}
  \action{
      \Lcal
      \sts{ {\lambda}\phi
            +
            \rho({ {\lambda}})\B{Q}
            +
            \eps
      }
  }{
    \psi
  }
  =
  \Scal^{\dprime}\sts{Q}
  \sts{
      {\lambda}\phi
            +
            \rho({ {\lambda}})\B{Q}
            +
            \eps,~
      \psi
  }.\numberthis\label{LS2d}
\end{align*}

Next, by \eqref{S3dQ} and \eqref{R2}, we have
\begin{multline*}
    \action{
        \Rcal_2
        \sts{Q,~ {\lambda}\phi
            +
            \rho({ {\lambda}})\B{Q}
            +
            \eps
        }
    }{
        \psi
    } =  -\frac{1}{2}
    \Scal^{\tprime}\sts{Q}
    \sts{
        {\lambda}\phi
            +
            \rho({ {\lambda}})\B{Q}
            +
            \eps
        ,~
        {\lambda}\phi
            +
            \rho({ {\lambda}})\B{Q}
            +
            \eps
        ,~
        \psi
    }.\numberthis\label{R2S3d}
\end{multline*}

At last,
 by \eqref{Rt}, we have
\begin{align*}
  \abs{ \tilde{\Rcal}\sts{Q,~\eta} }
  \leq
  C\sts{Q}\sts{ \abs{Q_x}\abs{\eta}^{2\sigma}
  +
  \abs{Q}^{2\sigma-2}\abs{\eta}^2\abs{\eta_x}
  +
  \abs{\eta}^{2\sigma}\abs{\eta_x}
  }.
  \numberthis\label{peRt}
\end{align*}
where $C\sts{Q}$ is a constant which only depends on
$Q$ and $\partial_x Q$.

Now let
\begin{multline*}
\Fcal\sts{Q,~\lambda,~\eps}
   =
  \Lcal
  \sts{ {\lambda}\phi
        +
        \rho({ {\lambda}})\B{Q}
        +
        \eps
  }
  -
  \Rcal_2
  \sts{Q,~ {\lambda}\phi
        +
        \rho({ {\lambda}})\B{Q}
        +
        \eps
  }
  -
  \tilde{\Rcal}
  \sts{Q,~ {\lambda}\phi
        +
        \rho({ {\lambda}})\B{Q}
        +
        \eps
  },
\end{multline*}
then $\Fcal\sts{Q,\lambda,\eps}$ is a polynomial of at least one degree with respect to $\lambda$ or $\eps$.   By  \eqref{rho}, \eqref{tp6}, we have
\begin{align*}
  \norm{
     {\lambda}\phi
     +
     \rho({ {\lambda}})\B{Q}
     +
     \eps
  }_{H^1\sts{\R}}
  \leq
  C\sts{Q}
  \bsts{ \abs{\lambda}+\norm{\eps}_{H^1\sts{\R}} },
\end{align*}
which implies that
\begin{align}\label{NLE}
\abs{  \left<\Fcal\sts{Q,~\lambda,~\eps}, ~\i\phi\right>}
+ \abs{\left<\Fcal\sts{Q,~\lambda,~\eps}, ~\i Q_x\right>}
+ \abs{\left<\Fcal\sts{Q,~\lambda,~\eps},~Q\right> }
\leq C\sts{Q}   \bsts{ \abs{\lambda}+\norm{\eps}_{H^1\sts{\R}} }.
\end{align}

Multiplying \eqref{epst} with $\i\phi$, $\i Q_x$ and $Q$ respectively, we have by \eqref{tp5},
\begin{multline*}
\lambda_t
  \baction{\i \sts{\phi
        +
        \dot{\rho}({\lambda})\B{Q}
            }
        }{\i\phi}
  -
    \sts{y_t-c}~
    \baction{\i \sts{Q + {\lambda}\phi
        +
        \rho({ {\lambda}})\B{Q}
        +
        \eps}_{x}
    }{\i\phi}
\\
  +
  \sts{\gamma_t+\omega}
  \baction{
    \sts{Q + {\lambda}\phi
        +
        \rho({ {\lambda}})\B{Q}
        +
        \eps}
  }{\i\phi}
=
  \baction{\Fcal\sts{Q,\lambda,\eps}}{\i\phi},
  \numberthis\label{tp7}
\end{multline*}
\begin{multline*}
  \lambda_t
  \baction{\i \sts{\phi
        +
        \dot{\rho}({\lambda})\B{Q}
            }
        }{\i Q_x}
  -
    \sts{y_t-c}~
    \baction{\i \sts{Q + {\lambda}\phi
        +
        \rho({ {\lambda}})\B{Q}
        +
        \eps}_{x}
    }{\i Q_x}
\\
  +\sts{\gamma_t+\omega}
  \baction{
    \sts{Q + {\lambda}\phi
        +
        \rho({ {\lambda}})\B{Q}
        +
        \eps}
  }{\i Q_x}
=
  \baction{\Fcal\sts{Q,\lambda,\eps}}{\i Q_x},
  \numberthis\label{tp8}
\end{multline*}
and
\begin{multline*}
  \lambda_t
  \baction{\i \sts{\phi
        +
        \dot{\rho}({\lambda})\B{Q}
            }
        }{Q}
  -
    \sts{y_t-c}~
    \baction{\i \sts{Q + {\lambda}\phi
        +
        \rho({ {\lambda}})\B{Q}
        +
        \eps}_{x}
    }{Q}
\\
  +\sts{\gamma_t+\omega}
  \baction{
    \sts{Q + {\lambda}\phi
        +
        \rho({ {\lambda}})\B{Q}
        +
        \eps}
  }{Q}
=
  \baction{\Fcal\sts{Q,\lambda,\eps}}{Q}.
  \numberthis\label{tp9}
\end{multline*}
Combining \eqref{tp7}, \eqref{tp8}, \eqref{tp9} with \eqref{NLE},   we can obtain the result.
\end{proof}

\section{Proof of Theorem \ref{thm}}\label{sect:main proof}
\begin{proof} We prove Theorem \ref{thm} by contradiction and divide the proof into several steps.

\begin{enumerate}[label=\emph{{Step \arabic*.}},ref=\emph{{Step \arabic*}}]

\item\label{ps1} \textit{Preparation of the initial data.}
Firstly, we can choose $0<\lambda_0\ll 1$ sufficiently small such that
$
  \Jcal\sts{u_0}=\Jcal\sts{Q},
$
where
\begin{equation*}
  u_0\sts{x} = Q\sts{x} + \lambda_0\phi \sts{x} +\widetilde{\rho}\sts{\lambda_0}\B{Q}\sts{x},
\end{equation*}
and $\widetilde{\rho}(\lambda)$ is determined by \Cref{lem:d1}.  It is easy to check that
\begin{equation*}
\norm{ u_0 -Q }_{H^1}\leq C \lambda_0 \ll 1.
\end{equation*}
Assume that the solitary wave $Q\sts{x-ct}\e^{\i\omega t}$  is orbitally stable for the degenerate case $\sigma\in \sts{1,~2}$ and $c=2z_0\sqrt{\omega}$ by contradiction, then for sufficiently small $\lambda_0$, we obtain that the solution $u\sts{t}$ of \eqref{gdnls} with initial data $u_0$ are global, and there exists $0<\delta \leq \bar{\delta}_1$ such that
\begin{equation*}
  u\sts{t}\in\Ucal\sts{Q~,~\delta}, \text{~~for all~~} t\geq 0.
\end{equation*}
Let $\rho$ be defined by \eqref{rho}. By Lemma \Cref{lem:d2} and the regularity argument in \cite{MartelM:Instab:gKdV}, there exist $\Ccal^1$ functions $\lambda$, $y$ and $\gamma$ with respect to $t$ such that the radiation term
\begin{equation}
\label{epsn}
  \eps\sts{t,x}
  =
  u\sts{t,x+y\sts{t}}\e^{\i\gamma\sts{t}}
  -
  \bsts{
    Q\sts{x}
    +
    {\lambda\sts{t}}\phi\sts{x}
    +
    \rho({ {\lambda}\sts{t}})\B{Q}\sts{x}
    }
\end{equation}
satisfies the equation
\begin{align*}
 \i  {\partial_t\eps} & ~-~
  \Lcal
  \bsts{ {\lambda(t)}\phi
        +
        \rho({ {\lambda(t)}})\B{Q}
        +
        \eps\sts{t}
  } \\
  = &\;
  -\i {\dot \lambda\sts{t}}~
  \bsts{\phi
        +
        \dot{\rho}({\lambda\sts{t}})\B{Q}
  }
   +
    \i \sts{\dot y\sts{t}-c}~\bsts{Q + {\lambda\sts{t}}\phi
        +
        \rho({ {\lambda\sts{t}}})\B{Q}
        +
        \eps\sts{t}}_{x}
        \\
        & \;
  -\sts{\dot \gamma\sts{t}+\omega}
  \bsts{Q + {\lambda\sts{t}}\phi
        +
        \rho({ {\lambda\sts{t}}})\B{Q}
        +
        \eps\sts{t}}
  \\ &\;
  -
  \Rcal_2
  \bsts{Q,~{\lambda\sts{t}}\phi
        +
        \rho({ {\lambda\sts{t}}})\B{Q}
        +
        \eps\sts{t}
  }
  -
  \tilde{\Rcal}
  \bsts{Q,~{\lambda\sts{t}}\phi
        +
        \rho({ {\lambda\sts{t}}})\B{Q}
        +
        \eps\sts{t}
  },  \numberthis \label{eq:eps}
\end{align*}
where  $\Lcal {\eta} $,  $ \Rcal_2\sts{Q,~\eta}$ and  $ \tilde{\Rcal}\sts{Q,~\eta}$ are defined in Lemma \ref{lem:eps},
 and for all $t\geq 0$, we have
\begin{equation}
\label{nvt}
  \action{\eps\sts{t}}{\i{Q}}=0,~~\action{\eps\sts{t}}{{Q}_{x}}=0,~~
  \action{\eps\sts{t}}{\phi}=0,
\end{equation}
and
 \begin{equation*}
      \norm{\eps\sts{t}}_{H^1}
      +
      \abs{\lambda\sts{t}}  + \abs{y\sts{t}}  + \abs{\gamma\sts{t}}
    \leq C \delta.
    \end{equation*}
By choosing $\lambda_0$ sufficiently small  if necessary, we have
\begin{equation*}
  \max\{1,~C\}~ \delta
  <
  \min
  \ltl{\bar{\lambda}_0,~\bar{\lambda}_2,~\bar{\lambda}_3,
    \bar{\lambda}_4,~\bar{\lambda}_5,~\bar{\lambda}_6
  },
\end{equation*}
and
\begin{equation*}
  \max\{1,~C\}~ \delta
  <
  \min
  \ltl{
    ~\bar{\delta}_2,~\bar{\delta}_4,~\bar{\delta}_5,~\bar{\delta}_6
  },
\end{equation*}
where $C$ is the constant in Lemma \ref{lem:d2}.
Hence, by Lemma \ref{lem:lyot}, we have  for all $t\in [0,\infty)$,
\begin{equation}
\label{nlyot}
  \abs{ \dot{\lambda}\sts{t} }
  +
  \abs{ \dot{y}\sts{t}-c }
  +
  \abs{ \dot{\gamma}\sts{t}+\omega }
  \leq
  C\sts{Q}\bsts{ \abs{\lambda\sts{t}}+\norm{\eps\sts{t}}_{H^1} }.
\end{equation}
By the conservation laws of mass and momentum, we have
\begin{equation*}
  \Jcal\sts{Q + {\lambda\sts{t}}\phi
        +
        \rho({ {\lambda\sts{t}}})\B{Q}+\eps\sts{t}
    }=\Jcal\sts{Q}.
\end{equation*}
By Lemma \Cref{lem:eBQ}, we can obtain
\begin{equation}\label{neeBQ}
  {\action{\eps\sts{t}}{\B{Q}}}
  =
  \bo{
    \abs{ {\lambda\sts{t}}}~\norm{\eps\sts{t}}_{H^1}
    +
    \norm{\eps\sts{t}}_{H^1}^2
  }
  +
  \so{ {\lambda^2\sts{t}}}.
\end{equation}

\item\label{ps2} \textit{Efficient control of $\lambda\sts{t}$ and $\eps\sts{t}$.} Combining the above estimates, we can obtain the following estimates about $\lambda\sts{t}$ and the radiation term $\eps\sts{t}$,
\begin{prop}
\label{prop:ele}
Let $u_0 = Q + \lambda_0\phi +\widetilde{\rho}\sts{\lambda_0}\B{Q}$
with $0<\lambda_0\ll 1$, and
\begin{equation*}
  \Jcal\sts{u_0}=\Jcal\sts{Q}
\end{equation*}
wehre $\widetilde{\rho}$ is determined by Lemma \ref{lem:d1}. Suppose $u\sts{t}$ is the global solution of \eqref{gdnls} with initial data $u_0$. If for all $t\in [0,\infty)$, there exist $\Ccal^1$ functions  $\lambda$, $y$ and $\gamma$ with respect to $t$ such that the radiation term
\begin{equation}
  \eps\sts{t,x}
  =
  u\sts{t,x+y\sts{t}}\e^{\i\gamma\sts{t}}
  -
  \bsts{
    Q\sts{x}
    +
    {\lambda\sts{t}}~\phi\sts{x}
    +
    \rho({ {\lambda}\sts{t}})~\B{Q}\sts{x}
    }
\end{equation}
 satisfies
\begin{equation}
\label{tvt}
  \action{\eps\sts{t}}{\i{Q}}=0,~~
  \action{\eps\sts{t}}{{Q}_{x}}=0,~~
  \action{\eps\sts{t}}{\phi}=0,\text{~~for all~~} t\geq 0,
\end{equation}
where $\rho\sts{\lambda}$ is defined by \eqref{rho}, then for any $t\geq 0$, we have
\begin{equation}
\label{lt}
  \lambda\sts{t}\geq \frac{1}{2}\lambda_0,
\end{equation}
and
\begin{equation}
\label{eet}
  \norm{\eps\sts{t}}_{H^1\sts{\R}}^2
  \leq
  -\frac{2}{\kappa}\mathbf{d}_{\bm\xi}^{\tprime}\cdot
  \lambda^3\sts{t},
\end{equation}
where $\mathbf{d}_{\bm\xi}^{\tprime}$ is defined by \eqref{ndeg}, and $\kappa$ is the constant defined in Lemma \ref{lem:coer}.
\end{prop}
\begin{proof} By the assumption $\mathbf{d}_{\bm\xi}^{\tprime}<0$ and the fact that $\lambda_0>0$, we have
\begin{equation*}
  \mathbf{d}_{\bm\xi}^{\tprime}\cdot \left(\lambda_0\right)^3<0.
\end{equation*}

Firstly, by $0<\lambda_0\ll 1$ and the Taylor series expression, we have for $t=0$ that
\begin{align*}
\Scal\sts{
    u_0
  }
  -
  \Scal\sts{
    Q
  }
=
&
  \Scal\sts{
    Q+\lambda_0\phi+\widetilde{\rho}\sts{\lambda_0}\B{Q}
  }
  -
  \Scal\sts{
    Q
  }
\\
  =
&
\frac{1}{2}
  \Scal^{\dprime}\sts{Q}
  \sts{
    \lambda_0\phi+\widetilde{\rho}\sts{\lambda_0}\B{Q},~
    \lambda_0\phi+\widetilde{\rho}\sts{\lambda_0}\B{Q}
  }
  \\
&
  +\frac{1}{6}
  \Scal^{\tprime}\sts{Q}
  \sts{
    \lambda_0\phi+\widetilde{\rho}\sts{\lambda_0}\B{Q},~
    \lambda_0\phi+\widetilde{\rho}\sts{\lambda_0}\B{Q},~
    \lambda_0\phi+\widetilde{\rho}\sts{\lambda_0}\B{Q}
  }
  \\
  &
  +
  \so{
    \norm{ \lambda_0\phi+\widetilde{\rho}\sts{\lambda_0}\B{Q}
    }_{H^1}^3
  }. \numberthis\label{ST1}
\end{align*}
where we used the fact that $\Scal'\sts{Q}=0$. By the expression of $\widetilde{\rho}$ in Lemma \ref{lem:d1}, we have
\begin{align*}
&\;
\Scal^{\dprime}\sts{Q}
\sts{
  \lambda_0\phi+\widetilde{\rho}\sts{\lambda_0}\B{Q},~
  \lambda_0\phi+\widetilde{\rho}\sts{\lambda_0}\B{Q}
}
\\
=
& \;
\left(\lambda_0\right)^2
\Scal^{\dprime}\sts{Q}
\sts{\phi,~\phi}
+
2\lambda_0\widetilde{\rho}\sts{\lambda_0}
\Scal^{\dprime}\sts{Q}
\sts{\phi,~ \B{Q}}
+
\widetilde{\rho}\sts{\lambda_0}^2\Scal^{\dprime}\sts{Q}
\sts{\B{Q},~\B{Q}}
\\
=
&
-
\left(\lambda_0\right)^2
\action{\B{Q}}{\phi}
-
2\lambda_0\widetilde{\rho}\sts{\lambda_0}
\action{\B{Q}}{\B{Q}}
+
\widetilde{\rho}\sts{\lambda_0}^2\Scal^{\dprime}\sts{Q}
\sts{\B{Q},~\B{Q}}
\\
=
&
-
2\lambda_0\widetilde{\rho}\sts{\lambda_0}
\action{\B{Q}}{\B{Q}}
+
\widetilde{\rho}\sts{\lambda_0}^2\Scal^{\dprime}\sts{Q}
\sts{\B{Q},~\B{Q}}
\\
=
&
\left(\lambda_0\right)^3
\action{\B{\phi}}{\phi}
+
\so{\abs{\lambda_0}^3},
\end{align*}
where we used the fact that $\action{\B{Q}}{\phi}=0$ in the fourth equality. Therefore, by inserting the above equality into
\eqref{ST1}, we have by \eqref{expdtp} that
\begin{align*}
\Scal\sts{
    u_0
  }
  -
  \Scal\sts{
    Q
  }
=
&  \left(
  \frac{1}{2}
    \action{\B{\phi}}{\phi}
  +\frac{1}{6}
  \Scal^{\tprime}\sts{Q}
  \sts{\phi,~\phi,~\phi} \right) \cdot \left(\lambda_0\right)^3
  +
  \so{\abs{\lambda_0}^3}
\\
  =
& \;
  \frac{1}{6}~
    \mathbf{d}_{\bm\xi}^{\tprime}\cdot \left( \lambda_0\right)^3
  +
  \so{\abs{\lambda_0}^3},
  \numberthis\label{eST1}
\end{align*}
where we used the fact that
$\rho\sts{\lambda}=\so{\abs{\lambda}}$
for $\abs{\lambda}\ll 1$.

Secondly, by Lemma \ref{lem:eBQ}, Lemma \ref{lem:tyl2} and Corollary \ref{cor:coer}, we know that for any $t\geq 0$,  there exists some $\kappa>0$ such that
\begin{align*}
  \Scal\sts{
    u\sts{t}
  }
  -
  \Scal\sts{
    Q
  }
 = & \;
  \Scal\sts{
    Q
    +
    \lambda\sts{t}\phi
    +
    \rho\sts{\lambda\sts{t}}\B{Q}+\eps\sts{t}
  }
  -
  \Scal\sts{
    Q
  }
\\
 = &\;
  \frac{1}{6}
    \mathbf{d}_{\bm\xi}^{\tprime}\cdot \left(\lambda\sts{t}\right)^3
  +
\Scal''\sts{Q}  \sts{\eps\sts{t},\eps\sts{t}}
+
  \so{\abs{\lambda\sts{t}}^3}
  +
  \so{ \norm{\eps\sts{t}}_{H^1}^2 }
     \\
  \geq  &\;
  \frac{1}{6}
    \mathbf{d}_{\bm\xi}^{\tprime}\cdot \left(\lambda\sts{t}\right)^3
  +
\frac{\kappa}{4}\norm{\eps\sts{t}}_{H^1}^2
+
  \so{\abs{\lambda\sts{t}}^3}
  +
  \so{ \norm{\eps\sts{t}}_{H^1}^2 }.
     \numberthis\label{eST2}
\end{align*}

Finally, by the conservation laws
of mass, momentum and energy, we have
\begin{equation*}
  \Scal\sts{u\sts{t}}=\Scal\sts{u_0},
  \text{~~for any ~~} t\geq 0.
\end{equation*}
Therefore, by \eqref{eST1}, \eqref{eST2} and the fact that
$\mathbf{d}_{\bm\xi}^{\tprime}<0$, we have
\begin{align*}
\frac{1}{24}~
    \mathbf{d}_{\bm\xi}^{\tprime}\cdot \left(\lambda_0\right)^3
\geq
& \;
  \frac{1}{6}~
    \mathbf{d}_{\bm\xi}^{\tprime}\cdot \left(\lambda_0\right)^3
  +
  \so{\abs{\lambda_0}^3}
\\
\geq
& \;
\frac{1}{6}~
    \mathbf{d}_{\bm\xi}^{\tprime}\cdot \left(\lambda\sts{t}\right)^3
  +
\frac{\kappa}{4}~\norm{\eps\sts{t}}_{H^1}^2
+
  \so{\abs{\lambda\sts{t}}^3}
  +
  \so{ \norm{\eps\sts{t}}_{H^1}^2 }
\\
\geq
&
\frac{1}{3}~
    \mathbf{d}_{\bm\xi}^{\tprime}\cdot \left(\lambda\sts{t}\right)^3
  +
\frac{\kappa}{6}~\norm{\eps\sts{t}}_{H^1}^2,
\end{align*}
which implies that
\begin{equation*}
  \lambda\sts{t}\geq \frac{1}{2}\lambda_0,
\text{~~ and ~~}
  \norm{\eps \sts{t}}_{H^1}^2
  \leq
  -
  \frac{2}{\kappa}~
  \mathbf{d}_{\bm\xi}^{\tprime}\cdot \left(\lambda\sts{t}\right)^3.
\end{equation*}
This concludes the proof of Proposition \Cref{prop:ele}.
\end{proof}

\item\label{ps3} \textit{Monotonicity formula.} Let us first define
\begin{equation}
\label{Phin}
  \Phi\sts{t,x} = \phi\sts{ x} + \alpha\lambda\sts{t } Q\sts{ x} + \beta\lambda\sts{ t}\i Q_x\sts{ x},
\end{equation}
where $\alpha$ and $\beta$ are chosen as following
\begin{equation*}
  \alpha
  =
  -\frac{
    \det\begin{bmatrix}
          \action{\phi}{\phi}  & \action{\i Q_x}{Q} \\
          \action{\i \phi_x}{\phi} & \action{\i Q_x}{\i Q_x} \\
         \end{bmatrix}
  }{
    \det\begin{bmatrix}
          \action{ Q}{  Q}  & \action{\i Q_x}{ Q} \\
          \action{\i Q_x}{ Q} & \action{\i Q_x}{\i Q_x} \\
         \end{bmatrix}
  },
\quad
  \beta
  =
  -\frac{
    \det\begin{bmatrix}
          \action{ Q}{ Q}  &  \action{\phi}{\phi} \\
          \action{\i Q_x}{ Q} &  \action{\i\phi_x}{\phi} \\
         \end{bmatrix}
  }{
    \det\begin{bmatrix}
          \action{ Q}{  Q}  & \action{\i Q_x}{ Q} \\
          \action{\i Q_x}{ Q} & \action{\i Q_x}{\i Q_x} \\
         \end{bmatrix}
  },
\end{equation*}
which are the solutions to the following system
\begin{align}
  \begin{bmatrix}
          \action{ Q}{  Q}  & \action{\i Q_x}{ Q} \\
          \action{\i Q_x}{ Q} & \action{\i Q_x}{\i Q_x} \\
  \end{bmatrix}
  \begin{bmatrix}
    \alpha \\
    \beta \\
  \end{bmatrix}
  =
  -
  \begin{bmatrix}
    \action{\phi}{\phi} \\
    \action{\i\phi_x}{\phi} \\
  \end{bmatrix},
\label{alpbet}
\end{align}
and imply that
\begin{align*}
&
  \baction{\B{Q}
}{
    \alpha Q
    + \beta\i Q_x}
=
-
\baction{\B{\phi}}{\phi}.
\numberthis\label{xialpbet}
\end{align*}

Next, we can define the Virial quantity as following
\begin{equation}
\label{I}
  \Iscr\sts{t} =\action{\i\eps\sts{t}}{\Phi\sts{t}}.
\end{equation}
By \eqref{nvt}, the straightforward inspection can show that
\begin{align*}
  \frac{\d}{\d t}\Iscr\sts{t}
  =&
  \action{\i\partial_t \eps}{\Phi\sts{t}}
  +\alpha\dot \lambda\sts{t} \action{\i\eps\sts{t} }{Q}
  +\beta \dot \lambda\sts{t} \action{\i\eps\sts{t} }{\i Q_x}
  \\
  =&
  \action{\i\partial_t \eps}{\Phi\sts{t}}.
  \numberthis\label{It}
\end{align*}

By inserting \eqref{eq:eps} into \eqref{It}, we have
\begin{align*}
\frac{\d}{\d t}\Iscr\sts{t}
=
&
-\dot \lambda\sts{t}
  \action{\i\sts{\phi
        +
        \dot{\rho}({\lambda(t)})\B{Q}
  }
}{\Phi\sts{t}}\numberthis \label{Ilt}
\\
&
+
\sts{\dot y\sts{t}-c}~
\action{\i\sts{Q + {\lambda\sts{t}}\phi
        +
        \rho({ {\lambda\sts{t}}})\B{Q}
        +
        \eps\sts{t}}_{x}
}{\Phi\sts{t}}
\numberthis \label{Iyt}
\\
&
-\sts{\dot \gamma\sts{t}+\omega}
\action{\sts{Q + {\lambda\sts{t}}\phi
        +
        \rho({ {\lambda\sts{t}}})\B{Q}
        +
        \eps\sts{t}}
}{\Phi\sts{t}}
\numberthis \label{Iot}
\\
&
+
\action{\Lcal
  \sts{ {\lambda\sts{t}}\phi
        +
        \rho({ {\lambda\sts{t}}})\B{Q}
        +
        \eps\sts{t}
  }
}{\Phi\sts{t}}
\numberthis \label{ILt}
\\
&
-
\action{\Rcal_2
  \sts{Q,~{\lambda\sts{t}}\phi
        +
        \rho({ {\lambda\sts{t}}})\B{Q}
        +
        \eps\sts{t}
  }
}{\Phi\sts{t}}
\numberthis \label{IR2t}
\\
&
-
\action{\tilde{\Rcal}
  \sts{Q,~{\lambda\sts{t}}\phi
        +
        \rho({ {\lambda\sts{t}}})\B{Q}
        +
        \eps\sts{t}
    }
}{\Phi\sts{t}}
\numberthis \label{ItRt}.
\end{align*}
We now estimate \eqref{Ilt}-\eqref{ItRt}.

\noindent\text{\bf Estimate of \eqref{Ilt}.} It follows from \eqref{dgf} that
\begin{align*}
& \; \baction{
    \i
    \sts{\phi
        +
        \dot{\rho}({\lambda\sts{t}})\B{Q}
    }
}{
        \Phi\sts{t}
}\\
=
&\;
\baction{
    \i
    \sts{\phi
        +
        \dot{\rho}({\lambda\sts{t}})\B{Q}
    }
}{
    \phi
    + \alpha\lambda\sts{t} Q
    + \beta\lambda\sts{t}\i Q_x
}
\\
=
& \;
\lambda\sts{t}
\baction{
    \i\phi
}{
    \alpha Q
    + \beta\i Q_x
}
+
\dot{\rho}({\lambda\sts{t}})
\baction{
    \i
        \B{Q}
}{
    \phi
}\\
=& \; 0.
\end{align*}
Therefore, we have
\begin{equation}
\label{eIlt}
  \eqref{Ilt}=0.
\end{equation}

\noindent\text{\bf Estimate of \eqref{Iyt}.} By \eqref{dgf} and \eqref{alpbet}, we have
\begin{align*}
&
\baction{
  \i\sts{Q + {\lambda\sts{t}}\phi
        +
        \rho({ {\lambda\sts{t}}})\B{Q}
        +
        \eps\sts{t}}_{x}
}{
    \Phi\sts{t}
}
\\
=
&
\baction{
  \i\sts{Q + {\lambda\sts{t}}\phi
        +
        \rho({ {\lambda\sts{t}}})\B{Q}
        +
        \eps\sts{t}}_{x}
}{
    \phi
    + \alpha\lambda\sts{t} Q
    + \beta\lambda\sts{t}\i Q_x
}
\\
=
&
\lambda\sts{t}
\baction{
  \i{Q}_{x}
}{
    \alpha Q
    + \beta\i Q_x
}
+
\lambda\sts{t}
\baction{
  \i{\phi}_{x}
}{
    \phi
}
+
\lambda\sts{t}^2
\baction{
  \i{\phi}_{x}
}{
    \alpha Q
    + \beta\i Q_x
}
\\
&
+
\rho({ {\lambda\sts{t}}})
\baction{
  \i\sts{
\B{Q}
        }_{x}
}{
    \phi
    + \alpha\lambda\sts{t} Q
    + \beta\lambda\sts{t}\i Q_x
}
+
\baction{
  \i\partial_x {\eps\sts{t}}
}{
    \phi
    + \alpha\lambda\sts{t} Q
    + \beta\lambda\sts{t} \i Q_x
}
\\
=
&
\lambda\sts{t}^2
\baction{
  \i{\phi}_{x}
}{
    \alpha Q
    + \beta\i Q_x
}
+
\rho({ {\lambda\sts{t}}})
\baction{
  \i\sts{
\B{Q}
        }_{x}
}{
    \phi
    + \alpha\lambda\sts{t} Q
    + \beta\lambda\sts{t}\i Q_x
}
\\
&
+
\baction{
  \i\partial_x {\eps\sts{t}}
}{
    \phi
    + \alpha\lambda\sts{t} Q
    + \beta\lambda\sts{t}\i Q_x
}.
\end{align*}
Thus, by \eqref{rho} and \eqref{eet}, we have
\begin{align*}
  \baction{
  \i\sts{Q + {\lambda\sts{t}}\phi
        +
        \rho({ {\lambda\sts{t}}})\B{Q}
        +
        \eps\sts{t}}_{x}
  }{
    \Phi\sts{t}
  }
=
\bo{ \lambda\sts{t}^2 +\norm{\eps\sts{t}}_{H^1} }
=
\so{ \abs{\lambda\sts{t}} },
\end{align*}
Combining the above estimate with \eqref{nlyot}
and \eqref{eet}, we have
\begin{equation}\label{eyot}
  \eqref{Iyt}
  =
  \bo{ \lambda\sts{t}^2 +\norm{\eps\sts{t}}_{H^1\sts{\R}} } \cdot
  \bo{ \abs{\lambda\sts{t}} +\norm{\eps\sts{t}}_{H^1\sts{\R}} }
  =
  \so{ {\lambda\sts{t}}^2 }.
\end{equation}

\noindent\text{\bf Estimate of \eqref{Iot}.} By \eqref{dgf}
and \eqref{alpbet}, we have
\begin{align*}
&
\baction{ {Q + {\lambda\sts{t}}\phi
        +
        \rho({ {\lambda\sts{t}}})\B{Q}
        +
        \eps\sts{t}}
}{\Phi\sts{t}}
\\
=
&
\baction{ {Q + {\lambda\sts{t}}\phi
        +
        \rho({ {\lambda\sts{t}}})\B{Q}
        +
        \eps\sts{t}}
}{\phi
    + \alpha\lambda\sts{t} Q
    + \beta\lambda\sts{t}\i Q_x}
\\
=
&
\lambda\sts{t}
\bsts{
    \baction{ {Q
        }
    }{
    \alpha Q
    + \beta\i Q_x}
    +
    \baction{ {\phi}}{\phi}
}
+
\baction{ {{\lambda\sts{t}}\phi
        }
}{
    \alpha\lambda\sts{t} Q
    + \beta\lambda\sts{t}\i Q_x}
\\
&
+
\baction{
        \rho({ {\lambda\sts{t}}})\B{Q}+\eps\sts{t}
}{\phi
    + \alpha\lambda\sts{t} Q
    + \beta\lambda\sts{t}\i Q_x}
\\
=
&
\baction{ {{\lambda\sts{t}}\phi
        }
}{
    \alpha\lambda\sts{t} Q
    + \beta\lambda\sts{t}\i Q_x}
+
\baction{
        \rho({ {\lambda\sts{t}}})\B{Q}+\eps\sts{t}
}{\phi
    + \alpha\lambda\sts{t} Q
    + \beta\lambda\sts{t}\i Q_x},
\end{align*}
which together with \eqref{dgf}, \eqref{rho}  implies that
\begin{align*}
{\baction{ {Q + {\lambda\sts{t}}\phi
        +
        \rho({ {\lambda\sts{t}}})\B{Q}
        +
        \eps\sts{t}}
        }{\Phi\sts{t}}
}
=
\bo{ \lambda\sts{t}^2 + \norm{\eps\sts{t}}_{H^1\sts{\R}} }.
\end{align*}
Combining the above estimate with \eqref{nlyot} and \eqref{eet}, we have
\begin{equation}
\label{eIot}
  \eqref{Iot}
  =
  \bo{ \lambda\sts{t}^2 +\norm{\eps\sts{t}}_{H^1\sts{\R}} }\cdot
  \bo{ \abs{\lambda\sts{t}} +\norm{\eps\sts{t}}_{H^1\sts{\R}} }
  =
  \so{ {\lambda\sts{t}}^2 }.
\end{equation}

\noindent\text{\bf Estimate of \eqref{ILt}.} By integrating by parts,
\eqref{dgf}, \eqref{expddp}, \eqref{rho}, \eqref{LS2d} and \eqref{xialpbet}, we have
\begin{align*}
& \eqref{ILt}\\
=
&
\baction{\Lcal
  \sts{ {\lambda\sts{t}}\phi
        +
        \rho({ {\lambda\sts{t}}})\B{Q}
        +
        \eps\sts{t}
  }
}{\phi
    + \alpha\lambda\sts{t} Q
    + \beta\lambda\sts{t}\i Q_x}
\\
=
&
-{\lambda\sts{t}}
\baction{\B{Q}
}{\phi
    + \alpha\lambda\sts{t} Q
    + \beta\lambda\sts{t}\i Q_x}
+
\rho({ {\lambda\sts{t}}})
\baction{
   {\B{Q}}
}{\Lcal\sts{\phi
    + \alpha\lambda\sts{t} Q
    + \beta\lambda\sts{t}\i Q_x}
}
\\
&
+
\baction{
  {\eps\sts{t}}
}{\Lcal\sts{\phi
    + \alpha\lambda\sts{t} Q
    + \beta\lambda\sts{t}\i Q_x}
}
\\
=
&
-{\lambda\sts{t}^2}
\baction{\B{Q}
}{
    \alpha Q
    + \beta\i Q_x}
-
\rho({ {\lambda\sts{t}}})
\baction{
   {\B{Q}}
}{{\B{Q}}
}
+
\lambda\sts{t}\rho({ {\lambda\sts{t}}})
\baction{
   {\B{Q}}
}{\Lcal\sts{
    \alpha Q
    + \beta\i Q_x}
}
\\
&
-
\baction{
  {\eps\sts{t}}
}{
    {\B{Q}}
}
+
\lambda\sts{t}\baction{
  {\eps\sts{t}}
}{\Lcal\sts{
    \alpha Q
    + \beta\i Q_x}
}
\\
=
&
{\lambda\sts{t}^2}
\baction{\B{\phi}}{\phi}
+
     \frac{\action{\B{\phi}}{\phi} }{2\action{\B{Q}}{\B{Q}}}\cdot \lambda\sts{t}^2\cdot
\baction{
   {\B{Q}}
}{{\B{Q}}
}
+
\lambda\sts{t}\rho({ {\lambda\sts{t}}})
\baction{
   {\B{Q}}
}{\Lcal\sts{
    \alpha Q
    + \beta\i Q_x}
}
\\
&
-
\baction{
  {\eps\sts{t}}
}{
    {\B{Q}}
}
+
\lambda\sts{t}\baction{
  {\eps\sts{t}}
}{\Lcal\sts{
    \alpha Q
    + \beta\i Q_x}
}.
\end{align*}
It follows that from \eqref{neeBQ} that
\begin{align*}
\eqref{ILt}
=
 \frac{3}{2}\lambda\sts{t}^2
 \baction{\B{\phi}}{\phi}
+
\bo{ \abs{ {\lambda\sts{t}}}~\norm{\eps\sts{t}}_{H^1} + \norm{\eps\sts{t}}_{H^1}^2 } +\so{ {\lambda\sts{t}}^2},
\end{align*}
which together with\eqref{eet} implies that
\begin{equation*}
\eqref{ILt}
=
\frac{3}{2}\lambda\sts{t}^2
 \baction{\B{\phi}}{\phi}
+
\so{ {\lambda\sts{t}}^2}.
\end{equation*}

\noindent\text{\bf Estimate of\eqref{IR2t}.} By \eqref{rho} and \eqref{R2S3d}, we have
\begin{align*}
&
\baction{\Rcal_2
  \sts{Q,~ {\lambda\sts{t}}\phi
        +
        \rho({ {\lambda\sts{t}}})\B{Q}
        +
        \eps\sts{t}
  }
}{
    \Phi\sts{t}
}
\\
=
&
\baction{\Rcal_2
  \sts{Q,~{\lambda\sts{t}}\phi
        +
        \rho({ {\lambda\sts{t}}})\B{Q}
        +
        \eps\sts{t}
  }
}{
    \phi
    + \alpha\lambda\sts{t} Q
    + \beta\lambda\sts{t}\i Q_x
}
\\
=
&
-
\frac{1}{2}
\Scal^{\tprime}\sts{Q}
    \sts{
    {\lambda\sts{t}}\phi
        +
        \rho({ {\lambda\sts{t}}})\B{Q}
        +
        \eps\sts{t},
    {\lambda\sts{t}}\phi
        +
        \rho({ {\lambda\sts{t}}})\B{Q}
        +
        \eps\sts{t},
    \phi
    + \alpha\lambda\sts{t} Q
    + \beta\lambda\sts{t}\i Q_x
    }
\\
=
&
-
\frac{1}{2}{\lambda\sts{t}}^2
\Scal^{\tprime}\sts{Q}\sts{\phi,\phi,\phi}
+
\bo{
    \abs{\lambda\sts{t}}^3
    +
    \abs{\lambda\sts{t}}\norm{\eps\sts{t}}_{H^1\sts{\R}}
    +
    \norm{\eps\sts{t}}_{H^1\sts{\R}}^2
},
\end{align*}
which together with \eqref{eet} implies that
\begin{equation}
\label{eR2t}
\eqref{IR2t}
=
   \frac{1}{2}{\lambda\sts{t}}^2
    \Scal^{\tprime}\sts{Q}
    \sts{\phi,\phi,\phi
    }
+
\so{\lambda\sts{t}^2}.
\end{equation}

\noindent\text{\bf Estimate of \eqref{ItRt}.} By \eqref{peRt}, it is easy to see that
\begin{align*}
&
\babs{
    \baction{
        \tilde{\Rcal}
        \sts{Q,{\lambda\sts{t}}\phi
            +
            \rho({ {\lambda\sts{t}}})\B{Q}
            +
            \eps\sts{t}
        }
    }{
        \Phi\sts{t}
    }
}
\\
\leq
&\;
C\sts{Q}
\int\abs{Q_x}
    \abs{{\lambda\sts{t}}\phi
            +
            \rho({ {\lambda\sts{t}}})\B{Q}
            +
            \eps\sts{t}}^{2\sigma}
    \abs{\Phi\sts{t}} \; dx
\\
&
  +
  C\sts{Q}
  \int
  \abs{Q}^{2\sigma-2}
  \abs{{\lambda\sts{t}}\phi
            +
            \rho({ {\lambda\sts{t}}})\B{Q}
            +
            \eps\sts{t}}^2
  \abs{\sts{{\lambda\sts{t}}\phi
            +
            \rho({ {\lambda\sts{t}}})\B{Q}
            +
            \eps\sts{t}}_x}~~
  \abs{\Phi\sts{t}} \; dx
\\
&
+
C\sts{Q}
\int\sts{
  \abs{{\lambda\sts{t}}\phi
            +
            \rho({ {\lambda\sts{t}}})\B{Q}
            +
            \eps\sts{t}}^{2\sigma}
  \abs{\sts{{\lambda\sts{t}}\phi
            +
            \rho({ {\lambda\sts{t}}})\B{Q}
            +
            \eps\sts{t}}_x}
  }\abs{\Phi\sts{t}} \; dx
\\
\leq
& \;
C\sts{Q}
\norm{\eps\sts{t}}_{H^1\sts{\R}}^{2\sigma}
+
\so{\lambda\sts{t}^2},
\end{align*}
which together with \eqref{eet} implies for $\sigma>1$ that
\begin{equation}\label{eItRt}
\eqref{ItRt}
=
\so{\lambda\sts{t}^2}.
\end{equation}
Therefore,  by summing up \eqref{eIlt}-\eqref{eItRt},
we can obtain from \eqref{exptdtp} that
\begin{align*}
\frac{\d}{\d t}\Iscr\sts{t}
=
  \frac{1}{2}~\mathbf{d}_{\bm\xi}^{\tprime}\cdot \lambda\sts{t}^2
  +
  \so{ \lambda\sts{t}^2 }.
  \numberthis\label{eIt}
\end{align*}

\item\label{ps4} \textit{Conclusion.}
On the one hand, by \eqref{epsn} and \eqref{Phin},
it is easy to see that
$\norm{\eps\sts{t}}_{H^1}$ and
$\norm{\Phi\sts{t}}_{H^1}$
are uniformly bounded with respect to $t$. Therefore, by
Cauchy-Schwarz's inequality, we have
\begin{equation}\label{tIbd}
  \abs{\Iscr\sts{t}}
  \text{~~uniformly bounded with respect to~~}
  t.
\end{equation}

On the other hand, since $\mathbf{d}_{\bm\xi}^{\tprime}<0$,
it follows from \eqref{lt} and \eqref{eIt} that
\begin{align*}
\frac{\d}{\d t}\Iscr\sts{t}
=
\frac{1}{2}~\mathbf{d}_{\bm\xi}^{\tprime}\cdot\lambda\sts{t}^2
+
\so{ \lambda\sts{t}^2 }
 \leq
 \frac{1}{4}~\mathbf{d}_{\bm\xi}^{\tprime}\cdot\lambda\sts{t}^2
\leq
\frac{1}{16}~\mathbf{d}_{\bm\xi}^{\tprime}\cdot \left(\lambda_0\right)^2,
\end{align*}
by integrating the above inequality
over $[0, ~t)$, we can obtain that
\begin{align*}
\Iscr\sts{t}
=
&
\Iscr\sts{0}
+
\int_{0}^{t}\Iscr^{\prime}\sts{s}\d s
\leq
\Iscr\sts{0}
+
\frac{1}{16}~\mathbf{d}_{\bm\xi}^{\tprime}\cdot\left(\lambda_0\right)^2 t,
\end{align*}
which means that
\begin{equation*}
  \lim_{t\to +\infty}\Iscr\sts{t}=-\infty,
\end{equation*}
this is a contradiction with \eqref{tIbd}.
\end{enumerate}

Above all, we complete the proof of Theorem \ref{thm}.
\end{proof}

\appendix
\section{Proof of Lemma \ref{lem:S3d}} \label{app:dds}
We will give the proof of Lemma \ref{lem:S3d} in this part and will drop the subscripts $\omega,~c$ for convenience if without confusion.

\begin{proof}
First, by the definition of $\Scal$, we have
\begin{equation}
	\Scal\sts{u}
    =
    \Re\int
    \sts{
        \frac{1}{2}\abs{{u}_x}^2
        +
        \frac{\omega}{2}\abs{u}^2
        +
        \frac{c}{2}{\i\bar{u}{u}_x}
        -
        \frac{1}{2\sigma + 2}
        {\i\abs{u}^{2\sigma}\bar{u}{u}_x}
    }.
\end{equation}
Since $\sigma>1$, it follows that
$\Scal:H^1\sts{\R}\mapsto\R$ is of class $\Ccal^2$,
therefore for any $u\in H^1\sts{\R}$ and $g,~h\in H^1\sts{\R}$, the straightforward calculations give that
\begin{equation*}
  \action{\Scal^{\prime}\sts{u}}{h}
  =
  \Re\int
  \sts{
  u_{x}\overline{{h}_{x}}
  +
  {\omega~} u\bar{h}
  +
  {c~}\i{u}_x\bar{h}
  -
    \i\abs{u}^{2\sigma}{u}_x\bar{h}
  }
  ,
\end{equation*}
and
\begin{align*}
  \Scal^{\dprime}\sts{u}\sts{h,g}
  =& \;
  \Re\int
  \sts{
   h_x\overline{g_x} + \omega~ h\bar{g}+c~h_x\bar{g}
   }
   \\
   & \;
   -
   \Re\int
   \sts{
   \i\abs{u}^{2\sigma}h_x\bar{g}
   +
   \i\sigma\abs{u}^{2\sigma-2}\bar{u}u_x h\bar{g}
   +
   \i\sigma\abs{u}^{2\sigma-2}{u}{u}_x \bar{h}\bar{g}
   },\numberthis\label{appS2d}
\end{align*}
and it is easy to see that
$
  \Scal^{\tprime}\sts{Q}\sts{h,g,f}
  =
\Ncal^{\tprime}\sts{Q}\sts{h,g,f}
$
if either $  \Scal^{\tprime}\sts{Q} $ or $\Ncal^{\tprime}\sts{Q}$ exists. In order to show that $\Scal$ is of class $\mathcal{C}^3$ at $Q$, we only need to prove that there
exists a linear operator $\tilde{\Tcal}:~H^1\sts{\R} \mapsto \sts{H^1\sts{\R}\times H^1\sts{\R}}'$ such that for any $g, ~h \in H^1\sts{\R}$, we have
\begin{equation*}
\abs{
  \Scal^{\dprime}\sts{Q+f}\sts{h,g}
  -
  \Scal^{\dprime}\sts{Q}\sts{h,g}
  -
  \tilde{\Tcal}\sts{f}\sts{h,g}
}
=
\so{\norm{f}_{H^1}}
\end{equation*}
as $\norm{f}_{H^1}$ goes to zero. By \eqref{appS2d}, we have
\begin{align*}
  &
  \Scal^{\dprime}\sts{Q+f}\sts{h,g}
  -
  \Scal^{\dprime}\sts{Q}\sts{h,g}
  \\
  =
  &
  \Re\int \left(
   h_x\overline{g_x} + \omega~ h\bar{g}+c~h_x\bar{g} \right) \d x
      -
   \Re\int
   \sts{
       \i\abs{{Q+f}}^{2\sigma}h_x\bar{g}
       +
       \i\sigma
       \abs{{Q+f}}^{2\sigma-2}\sts{\overline{Q+f}}\sts{Q+f}_x h\bar{g}
    } \d x
    \\
    &
    -
   \Re\int
   \sts{
       \i\sigma
       \abs{{Q+f}}^{2\sigma-2}{\sts{Q+f}}
       \sts{{Q}+{f}}_x \bar{h}\bar{g}
   }\d x
   -
   \Re\int
   \sts{
       h_x\overline{g_x} + \omega~ h\bar{g}+c~h_x\bar{g}
   } \d x
   \\
   &
   +
   \Re\int
   \sts{
       \i\abs{Q}^{2\sigma}h_x\bar{f}
       +
       \i\sigma\abs{Q}^{2\sigma-2}\bar{Q}Q_x h\bar{f}
       +
       \i\sigma\abs{Q}^{2\sigma-2}{Q}{Q}_x \bar{h}\bar{f}
   }\d x
\\
=
&
   -
   \Re\int
   \sts{
       \i\abs{{Q+f}}^{2\sigma}h_x\bar{g}
       -
       \i\abs{Q}^{2\sigma}h_x\bar{g}
   } \d x
   \numberthis \label{a1}
   \\
   &
   -
   \Re\int
    \i\sigma
   \sts{
       \abs{{Q+f}}^{2\sigma-2}\sts{\overline{Q+f}}\sts{Q+f}_x h\bar{g}
       -
       \abs{Q}^{2\sigma-2}\bar{Q}Q_x h\bar{g}
   }\d x
   \numberthis \label{a2}
   \\
   &
   -
   \Re\int
   \i\sigma
   \sts{
        \abs{{Q+f}}^{2\sigma-2}\sts{{Q+f}}\sts{ {Q}+{f}}_x \bar{h}\bar{g}
        -
        \abs{Q}^{2\sigma-2}{Q}{Q}_x \bar{h}\bar{g}
   } \d x \numberthis \label{a3}
\end{align*}

Firstly, since for any $z\in\C$,
\begin{align*}
  \bigg\vert{
    \abs{1+z}^{2\sigma}-1-\sigma{z}-\sigma\bar{z}
  }\bigg\vert
  \leq C\sts{ \abs{z}^2 +\abs{z}^{2\sigma} },
\end{align*}
where $C$ is a constant independent of $z$, we have
\begin{align*}
& \;
\bigg\vert{
    \eqref{a1}+
    \Re\int
    \sts{
       \i\sigma\abs{Q}^{2\sigma-2}{Q}\bar{f}
       +
       \i\sigma\abs{Q}^{2\sigma-2}\bar{Q}{f}
    }h_x\bar{g}
}\d x \bigg\vert
\\
=
&\;
  \bigg\vert{
  \Re\int
  \sts{
       \i\abs{{Q+f}}^{2\sigma}
       -
       \i\abs{Q}^{2\sigma}
       -
       \i\sigma\abs{Q}^{2\sigma-2}{Q}\bar{f}
       -
       \i\sigma\abs{Q}^{2\sigma-2}\bar{Q}{f}
  }
  h_x\bar{g}
  } \d x \bigg\vert
  \\
  \leq
  & \;
  C \norm{g}_{H^1 }\norm{h}_{H^1 }
  \sts{\norm{f}_{H^1}^2 +\norm{f}_{H^1}^{2\sigma}},
  \numberthis\label{a1l}
\end{align*}
where we used the fact that $\abs{Q}$ vanishes nowhere.
Let us denote
\begin{equation*}
  \tilde{\Tcal}_1\sts{f}\sts{h,g}
  =
  -
  \Re\int
  \sts{
       \i\sigma\abs{Q}^{2\sigma-2}{Q}\bar{f}
       +
       \i\sigma\abs{Q}^{2\sigma-2}\bar{Q}{f}
    }h_x\bar{g} \d x.
\end{equation*}

Secondly, notice that
\begin{align*}
    \eqref{a2}
   =
   &
   -
   \Re\int
   \i\sigma
   \sts{
       \abs{{Q+f}}^{2\sigma-2}
       \sts{\overline{Q+f}}\sts{Q+f}_x h\bar{g}
       -
       \abs{Q}^{2\sigma-2}\bar{Q}Q_x h\bar{g}
   }\d x
   \\
   =
   &
   -
   \Re\int
   \i\sigma
   \sts{
       \abs{{Q+f}}^{2\sigma-2}\sts{\overline{Q+f}}\sts{Q+f}_x h\bar{g}
       -
       \abs{{Q+f}}^{2\sigma-2}\sts{\overline{Q+f}}Q_x h\bar{g}
   } \d x
   \\
   &
   -
   \Re\int
   \i\sigma
   \sts{
       \abs{{Q+f}}^{2\sigma-2}\sts{\overline{Q+f}}Q_x h\bar{g}
       -
       \abs{Q}^{2\sigma-2}\bar{Q}Q_x h\bar{g}
   } \d x
   \\
   =
   &
   -
   \Re\int
   \i\sigma
    \abs{{Q+f}}^{2\sigma-2}\sts{\overline{Q+f}}
        {f}_x h\bar{g} \d x
   \numberthis \label{a21}
   \\
   &
   -
   \Re\int
   \i\sigma
   \sts{
       \abs{{Q+f}}^{2\sigma-2}\sts{\overline{Q+f}}
       -
       \abs{Q}^{2\sigma-2}\bar{Q}
   }Q_x h\bar{g} \d x.
   \numberthis \label{a22}
\end{align*}
On the one hand, since
\begin{align*}
  \bigg\vert
  {
    \abs{1+z}^{2\sigma-2}\sts{1+z}-1
  }
  \bigg\vert
  \leq C\sts{ \abs{z} +\abs{z}^{2\sigma-1} },
\end{align*}
where $C$ is a constant independent of $z$, we have
\begin{align*}
  \bigg\vert{
    \eqref{a21}
    +
    \Re\int
    \i\sigma
    \abs{{Q}}^{2\sigma-2} {\bar{Q}}
        {f}_x h\bar{g}
  }\d x
  \bigg\vert
  \leq \; C~\norm{g}_{H^1\sts{\R}}\norm{h}_{H^1\sts{\R}} \sts{\norm{f}_{H^1\sts{\R}}^2 +\norm{f}_{H^1\sts{\R}}^{2\sigma}}.
\numberthis\label{tp10}
\end{align*}
On the other hand, since
\begin{align*}
  \bigg\vert
  {
    \abs{1+z}^{2\sigma-2}\sts{1+z}-1
    - \sigma {z} - \sts{\sigma-1}\bar{z}
  }
  \bigg\vert
  \leq C\sts{ \abs{z}^2 +\abs{z}^{2\sigma-1} },
\end{align*}
where $C$ is a constant independent of $z$, which together with \eqref{QDQ} implies that
\begin{align*}
&\;
  \bigg\vert
    \eqref{a22}
    +
    \Re\int
    \i
    \sts{
    \sigma^2
    \abs{Q}^{2\sigma-2}\bar{f}
    +
    \sigma\sts{\sigma-1}
    \abs{Q}^{2\sigma-4}\bar{Q}^2 {f}
    }
    Q_x  h\bar{g} \d x
  \bigg\vert
\\
\leq
& \;
C
\sts{
    \int\abs{Q}^{2\sigma-3}~\abs{Q_x}~\abs{f}^2~\abs{g}~\abs{h} \d x  +
    \int\abs{Q_x}~\abs{f}^{2\sigma-1}~\abs{g}~\abs{h} \d x
}
\\
\leq
& \;
C\norm{g}_{H^1\sts{\R}}\norm{h}_{H^1\sts{\R}}
\sts{
    \frac{1}{c}\norm{Q}_{\infty}^{2\sigma-2}
    \norm{f}_{H^1\sts{\R}}^2
    +
    \frac{1}{c}\norm{Q}_{\infty}
    \norm{f}_{H^1\sts{\R}}^{2\sigma-1}
}.
\numberthis\label{tp11}
\end{align*}
Now, let us denote
\begin{align*}
\tilde{\Tcal}_2\sts{f}\sts{h,g}
=
-
\Re\int
  \sts
  {
    \i\sigma
    \abs{{Q}}^{2\sigma-2} {\bar{Q}}
        {f}_x h\bar{g}
  +
    \i
    \sigma^2
    \abs{Q}^{2\sigma-2}Q_x \bar{f} h\bar{g}
    +
    \i
    \sigma\sts{\sigma-1}
    \abs{Q}^{2\sigma-4}\bar{Q}^2 Q_x {f}h\bar{g}
  } \d x.
\end{align*}
By \eqref{tp10} and \eqref{tp11}, we have
\begin{align*}
\bigg\vert{
  \eqref{a2}
  -
  \tilde{\Tcal}_2\sts{f}\sts{h,g}
}
\bigg\vert
\leq\;
C\sts{Q}\norm{g}_{H^1\sts{\R}}\norm{h}_{H^1\sts{\R}}
\sts{
    \norm{f}_{H^1\sts{\R}}^2
    +
    \norm{f}_{H^1\sts{\R}}^{2\sigma}
}.
\numberthis\label{a2l}
\end{align*}

%

Thirdly, we denote
\begin{align*}
\tilde{\Tcal}_3\sts{f}\sts{h,g}
=
-
    \Re\int
  \sts
  {
    \i\sigma
    \abs{{Q}}^{2\sigma-2} {{Q}}
         {f}_x \bar{h}\bar{g}
  +
    \i
    \sigma^2
    \abs{Q}^{2\sigma-2} {Q}_x f \bar{h}\bar{g}
    +
    \i
    \sigma\sts{\sigma-1}
    \abs{Q}^{2\sigma-4}{Q}^2 {Q}_x \bar{f}\bar{h}\bar{g}
  } \d x.
\end{align*}
By the similar argument as above,
we can show that
\begin{align*}
\bigg\vert{
  \eqref{a3}
  -
  \tilde{\Tcal}_3\sts{f}\sts{h,g}
}
\bigg\vert
\leq
\;
C\sts{Q}\norm{g}_{H^1\sts{\R}}\norm{h}_{H^1\sts{\R}}
\sts{\norm{f}_{H^1\sts{\R}}^2 +\norm{f}_{H^1\sts{\R}}^{2\sigma}}.
\numberthis\label{a3l}
\end{align*}

Finally, we denote
\begin{equation}\label{T}
  \Tcal\sts{f,g,h}
  =
  \tilde{\Tcal}_1\sts{f}\sts{h,g}
  +
  \tilde{\Tcal}_2\sts{f}\sts{h,g}
  +
  \tilde{\Tcal}_3\sts{f}\sts{h,g},
\end{equation}
then by \eqref{a1l}, \eqref{a2l} and \eqref{a3l}, we
obtain
\begin{align*}
\abs{
    \Scal^{\dprime}\sts{Q+f}\sts{h,g}
    -
    \Scal^{\dprime}\sts{Q}\sts{h,g}
    -
    \Tcal\sts{f,g,h}
} = \so{\norm{f}_{H^1}}
\end{align*}
as $\norm{f}_{H^1}$ goes to zero.
Since $\sigma>1$, we have
\begin{equation*}
  \Tcal\in
  \sts{H^1\sts{\R} \times H^1\sts{\R} \times H^1\sts{\R}}'.
\end{equation*}
Therefore, $\Scal^{\tprime}\sts{Q}$ exists, and
$
  \Scal^{\tprime}\sts{Q}\sts{h,g,f}
  =
  \Tcal\sts{f,h,g}.
$ This completes the proof of Lemma \ref{lem:S3d}.
\end{proof}

\subsection*{Acknowledgements.}
The authors have been partially supported by the NSF grant of China (No. 11671046, No. 11671047) and also partially
supported by Beijing Center of Mathematics and Information
Interdisciplinary Science.



\end{document}